\documentclass[journal,twoside,web]{IEEEtran}
\usepackage{cite}
\usepackage{amsmath,amssymb,amsfonts}

\usepackage{amsthm}
\usepackage{graphicx}
\usepackage{algorithm,algorithmic}
\usepackage{hyperref}
\hypersetup{hidelinks=true}
\usepackage{textcomp}
\usepackage{bbm}

\usepackage{comment}
\usepackage{mathtools}
\graphicspath{ {./images/} }
\usepackage{balance}

\newtheorem{theorem}{Theorem}

\newtheorem{remark}{Remark}
\newtheorem{proposition}{Proposition}
\newtheorem{definition}{Definition}
\newtheorem{assumption}{Assumption}
\newtheorem{corollary}{Corollary}

\begin{document}

\title{Multivariable Extremum Seeking for Locally Lipschitz Objectives}

\author{Alan Williams%
\thanks{This work was supported by the U.S. Department of Energy through the Los Alamos National Laboratory. Los Alamos National Laboratory is operated by Triad National Security, LLC, for the National Nuclear Security Administration of U.S. Department of Energy (Contract No. 89233218CNA000001).}
\thanks{Alan Williams is with the Accelerator Operations and Technology - Instrumentation and Controls (AOT-IC) Group, Adaptive Machine Learning Team at Los Alamos National Laboratory, Los Alamos, NM 87545, USA (e-mail: awilliams@lanl.gov).}
}

\maketitle

%%%%%%%%%%%%%%%%%%%%%%%%%%%%%%%%%%%%%%%%%%%%%%%%%%%%%%%%%%%%%%%%%%%%%%%%%%%%%%%%
\begin{abstract}
Classical extremum seeking (ES) is commonly interpreted as approximating
gradient descent, but this interpretation is less clear for nonsmooth
objectives in the continuous-time multivariable setting. We propose a
minimal modification of the classical multivariable
perturbation--demodulation architecture: rationally independent
perturbation frequencies and matched demodulation signals. For any locally
Lipschitz static objective, the Kronecker--Weyl theorem shows that, at every
fixed perturbation amplitude, the long-time averaged dynamics are
exactly the negative gradient of a kernel-smoothed objective. Because
rationally independent frequencies render the perturbation and
demodulation signals nonperiodic, we employ general averaging theory rather than periodic averaging theory. If the gradient flow of the smoothed objective is
globally uniformly asymptotically stable, then the ES dynamics are
practically globally uniformly asymptotically stable. We also derive a
general matching condition relating the perturbation occupation density,
demodulation signal, and smoothing kernel, yielding a family of alternative
designs. Numerical examples include a nonsmooth objective function, which may be interpreted as the penalty function of a nonlinear program, and the Rastrigin function, for which smoothing eliminates all undesired local minima.
\end{abstract}


%%%%%%%%%%%%%%%%%%%%%%%%%%%%%%%%%%%%%%%%%%%%%%%%%%%%%%%%%%%%%%%%%%%%%%%%%%%%%%%%
\section{Introduction}
Consider the familiar continuous-time multivariable extremum seeking (ES) law
\begin{equation*}
\dot{\hat x}
=
-kJ\bigl(\hat x +S(t)\bigr)M(t).
\end{equation*}
In the classical sinusoidal design, the components of the perturbation and demodulation signals are \(S_i(t)=a\sin(\omega \hat \omega_i t)\) and \(M_i(t)=(2/a)\sin(\omega \hat \omega_i t)\) for $i=1, \dots, n$ with some perturbation amplitude $a>0$, adaptation gain $k>0$, perturbation frequencies $\hat \omega_i >0$, and frequency scale $\omega>0$. The sinusoidal perturbation \(S(t)\) explores the objective near the current parameter estimate $\hat x(t)$, and multiplication by \(M(t)\) extracts information to drive the estimate in a gradient descent direction. Standard multivariable analyses commonly choose the frequencies $\hat \omega_i$ so that \(S(t)\) has a common period $T$ and then use a small-$a$ Taylor expansion of $J$ near $\hat x$ before averaging the dynamics to come to the interpretation of the averaged dynamics as 
\begin{equation*}
    \frac{\mathrm{d}z}{\mathrm{d}\tau}
     = -\varepsilon \bigl( \nabla J(z) + O(a)\bigr).
\end{equation*}

For a locally Lipschitz objective, however, \(\nabla J\) may not exist
everywhere. A natural hope is that the averaged ES dynamics can still
be interpreted as a weighted average of nearby gradients wherever they exist. Rademacher's theorem makes this plausible: local Lipschitz continuity implies that \(\nabla J\) exists almost everywhere. The conventional periodic multivariable perturbation $S(t)$ presents an obstacle, however: in $n>1$ dimensions its trajectory only traverses a closed one-dimensional curve in $[-a, a]^n$, and cannot densely sample points in the neighborhood around the parameter estimate.

We make two minimal modifications to this classical architecture. First,
we choose rationally independent relative perturbation frequencies---for example, \(\hat\omega_1=1\) and \(\hat\omega_2=\sqrt{2}\) for $n=2$---so that the perturbation trajectory $S(t)$ densely explores the full perturbation region. The resulting perturbation trajectory is nonperiodic and dense in \([-a,a]^n\), so it comes arbitrarily close to every point in the exploration region as $t \to \infty$, and samples that region with an occupation density. Second, we match the demodulation signal to this occupation density so that every component of the averaged vector field uses the same smoothing kernel. The Kronecker--Weyl theorem then yields the key result of this paper: at every fixed perturbation amplitude, the averaged vector field is exactly the negative gradient of a smoothed objective,
\[
    \frac{\mathrm{d}z}{\mathrm{d}\tau}
    =
    -\varepsilon\nabla J_a(z),
\]
where
\[
    J_a(x)
    =
    \int_{[-1,1]^n}J(x+au)\kappa(u)\,\mathrm{d}u
\]
and \(\kappa \geq 0\) is a smoothing kernel of unit mass. Differentiability of \(J\) is not required by the Kronecker--Weyl theorem or by the general averaging results used in our analysis. In one dimension, our proposed sinusoidal design coincides with classical ES.

\subsection{Literature}

Extremum seeking (ES) is a model-free adaptive control technique that
adjusts system inputs online to optimize a measured performance output.
Stability of classical sinusoidal ES was established using averaging and
singular perturbations \cite{krstic2000stability} and subsequently treated
systematically in \cite{ariyur2003real}. Later work addressed nonlocal and
semiglobal stability \cite{tan2006non}, global convergence despite local
extrema \cite{tan2009global}, Newton-based schemes
\cite{ghaffari2012multivariable}, stochastic ES \cite{liu2015stochastic},
and constraints and safety \cite{williams2026local, williams2024semiglobal, williams2026generalized}. Smooth and nonsmooth multivariable ES based on nonlinear programming was considered
early in \cite{teel2001solving}. More recent extensions include hybrid and
accelerated architectures \cite{poveda2017framework,poveda2021robust},
time-varying objectives \cite{grushkovskaya2017extremum}, distributed
derivative-free optimization \cite{mimmo2024extremum}, delays
\cite{tsubakino2023extremum}, fixed- and prescribed-time convergence
\cite{poveda2021nonsmooth,yilmaz2024prescribed}, vanishing step sizes
\cite{grushkovskaya2024step}, uniform nonconvex guarantees
\cite{mimmo2024uniform}, and higher-order Lie-bracket averaging
\cite{pokhrel2026higher}. Game-theoretic extensions include model-free and
distributed Nash-equilibrium seeking
\cite{frihauf2012nash,stankovic2012distributed} and a nested architecture
for Stackelberg-equilibrium seeking \cite{ratto2026nested}. A comprehensive
historical and theoretical overview is provided in \cite{scheinker2024100}.

Work most closely related to ours concerns nonsmooth ES and spatial
smoothing. Lie-bracket approximations characterize broad classes of ES
systems \cite{durr2013lie} and have been extended to vector fields that may
fail to be differentiable at a point \cite{scheinker2014non}. For locally
Lipschitz objectives, \cite{suttner2023nonsmooth} uses randomly sampled
directions to obtain a stochastic gradient-like approximation, while
deterministic circular source seeking recovers the gradient of a
disk-averaged objective through the divergence theorem
\cite{suttner2024overcoming}. The higher-dimensional spherical-perturbation
design in \cite{suttner2026non} assumes a smooth objective; for \(n\geq3\),
averaging at fixed \(k\) produces a \(k\)-dependent field, with the gradient
of the ball-averaged objective recovered as \(k\to\infty\). In contrast, for arbitrary locally Lipschitz objectives, our fixed design has long-time averaged dynamics (and not periodically averaged) exactly equal to the negative gradient of an explicitly kernel-smoothed objective and admits multiple perturbation--kernel choices.
Nonsmooth high-order averaging instead yields generalized-gradient averaged
dynamics for a class of nonsmooth objectives
\cite{abdelfattah2026nonsmooth}. Experimental work has also shown that
harmonic frequency relations can periodically produce large plant
disturbances \cite{scheinker2013extremum}, motivating our use of rationally
independent frequencies. Related zeroth-order optimization methods use only function evaluations \cite{liu2020primer}. Random perturbations can provide unbiased estimates of gradients of smoothed objectives \cite{flaxman2004online}; our
design obtains such a gradient deterministically through long-time
averaging and selects the smoothing kernel through matched
perturbation--demodulation signals.

\subsection{Contributions}
The contributions of this paper are summarized as follows. 
\begin{enumerate}
    \item We introduce a multivariable ES design with
    rationally independent perturbation frequencies and matched demodulation
    signals. For an arbitrary locally Lipschitz objective, its long-time average dynamics are shown to be exactly the negative gradient of a single smoothed objective defined by a common kernel.
    \item We establish a practical global uniform asymptotic stability
    result for the proposed system when the gradient flow of the smoothed objective is globally uniformly asymptotically stable, allowing the
    analysis to cover both familiar convex settings and nonconvex objectives
    whose undesired stationary points are removed by smoothing. The result also separates the smoothing effect of the perturbation
    amplitude \(a\) from the averaging error governed by the time-scale ratio
    \(\varepsilon=k/\omega\).
    \item We derive a general matching relation between the spatial
    occupation density generated by a perturbation signal, the demodulation signal, and a
    desired smoothing kernel. This relation explains the proposed
    sinusoidal design, identifies why the classical multivariable
    demodulator does not generally yield the gradient of a single smoothed
    objective, and provides a constructive method for selecting other
    perturbation--kernel pairs. We present a sinusoidal perturbation design and a triangle-wave perturbation design with two valid matched demodulation signals.
\end{enumerate}

\textbf{Notation: }
A continuous function $\alpha:\mathbb{R}_{\ge 0}\to\mathbb{R}_{\ge 0}$ is of class $\mathcal K$ if it is strictly increasing and $\alpha(0)=0$. A continuous function $\beta:\mathbb{R}_{\ge 0}\times\mathbb{R}_{\ge 0}\to\mathbb{R}_{\ge 0}$ is of class $\mathcal{KL}$ if, for each fixed $s\ge 0$, the map $r\mapsto \beta(r,s)$ is of class $\mathcal K$, and for each fixed $r>0$, the map $s\mapsto \beta(r,s)$ is decreasing with $\beta(r,s)\to 0$ as $s\to\infty$. 
For $x \in \mathbb{R}^n$, we denote the Euclidean norm by
$\|x\|=\|x\|_2=\sqrt{x_1^2+\cdots+x_n^2}$, and \(x_{-i}\in\mathbb{R}^{n-1}\) denotes the vector obtained by removing the \(i\)th component of \(x\). The set of rationals is denoted by $\mathbb{Q}$ and $\ell \in \mathbb{Q}^n$ denotes a vector with each $\ell_i \in \mathbb{Q}$ for $i = 1 , \ldots, n$. All integrals over subsets of Euclidean space are Lebesgue integrals, and ``almost everywhere'' refers to Lebesgue measure.

\section{A Motivating Discussion}
\label{sec:motivating_discussion}

Consider the classical one-dimensional extremum seeking scheme
\begin{equation}
    \dot{\hat x}
    =
    -k\,J\bigl(\hat x+a\sin(\omega t)\bigr)
    \frac{2}{a}\sin(\omega t),
    \label{eq:ES-classical-t}
\end{equation}
where $\hat x\in\mathbb{R}$ is the parameter estimate, $k>0$ is the
adaptation gain, $a>0$ is the perturbation amplitude, $\omega>0$ is the perturbation
frequency, and $J:\mathbb{R}\to\mathbb{R}$ is the objective to be
minimized. The usual interpretation is that the perturbation
$a\sin(\omega t)$ probes nearby objective values, while multiplication by
$(2/a)\sin(\omega t)$ extracts information that drives $\hat x$ in a
descent direction. We now show that, for any locally Lipschitz objective \(J\), the averaged dynamics of this one-dimensional scheme are exactly the negative-gradient flow of a kernel-smoothed version of \(J\).

Introduce the fast time $\tau=\omega t$ and the parameter
$\varepsilon=k/\omega$. Then \eqref{eq:ES-classical-t} becomes
\begin{equation}
    \frac{\mathrm{d}\hat x}{\mathrm{d}\tau}
    =
    -\varepsilon
    J\bigl(\hat x+a\sin\tau\bigr)
    \frac{2}{a}\sin\tau.
    \label{eq:ES-classical-tau}
\end{equation}
The average of the vector field multiplying $-\varepsilon$ is
\begin{equation}
    \bar F(z)
    =
    \frac{1}{\pi a}
    \int_0^{2\pi}
    J\bigl(z+a\sin\tau\bigr)
    \sin\tau\,\mathrm{d}\tau,
    \label{eq:Fbar-first}
\end{equation}
and the averaged dynamics in fast time are
\begin{equation}
    \frac{\mathrm{d} z}{\mathrm{d}\tau}
    =
    -\varepsilon\bar F(z).
    \label{eq:avg-dyn-first}
\end{equation}

For a smooth objective, standard analysis techniques approximate $J$ through a
small-$a$ Taylor expansion about $z$. Because $J$ is only locally Lipschitz and need not be differentiable at $z$, this expansion is unavailable, so we must characterize the integral in \eqref{eq:Fbar-first} by some other means.

We take an alternative route: since \(J\) is locally Lipschitz, the composition
\(\tau\mapsto J(z+a\sin\tau)\) is Lipschitz and absolutely
continuous on \([0,2\pi]\), its
derivative exists almost everywhere\footnote{An absolutely continuous function on $\mathbb{R}$ is differentiable almost everywhere with respect to the Lebesgue measure \cite[Theorem~3.35]{folland1999real}. Famously, Rademacher also proved that a Lipschitz continuous function on $\mathbb{R}^n$ is differentiable almost everywhere \cite[Theorem~3.2]{evans1991measure}.}. So, we have
\[
    \frac{\mathrm{d}}{\mathrm{d}\tau}
    J\bigl(z+a\sin\tau\bigr)
    =
    aJ'\bigl(z +a\sin\tau\bigr)\cos\tau
\]
for almost every $\tau$. Applying integration by parts to
\eqref{eq:Fbar-first}, with the derivative $J'$ understood almost
everywhere, gives
\begin{equation}
    \bar F(z)
    =
    \frac{1}{\pi}
    \int_0^{2\pi}
    J'\bigl(z+a\sin\tau\bigr)
    \cos^2\tau\,\mathrm{d}\tau.
    \label{eq:Fbar-derivative-form}
\end{equation}
The boundary term is zero since \[\bigl. J(z + a \sin(\tau)) \cos(\tau) \bigr|^{2 \pi}_{0} =0 .\]

Since the integrand is periodic, we split the integral in \eqref{eq:Fbar-derivative-form} from $[0,2 \pi]$ into the sum of two integrals from $[-\pi/2, \pi/2]$ and $[\pi/2, 3\pi/2]$. To expose the spatial averaging performed by the perturbation, set
\(u=\sin\tau\) to arrive at
\begin{equation}
    \bar F(z)
    =
    \frac{2}{\pi}
    \int_{-1}^{1}
    J'(z+au)\sqrt{1-u^2}\,\mathrm{d}u.
    \label{eq:Fbar-u-form}
\end{equation}
Define the normalized semicircle kernel
\begin{equation}
    \kappa(u)
    :=
    \begin{cases}
        \dfrac{2}{\pi}\sqrt{1-u^2}, & |u|\leq 1,\\[1.2ex]
        0, & |u|>1.
    \end{cases}
    \label{eq:motivating-kappa}
\end{equation}
This kernel is nonnegative, even, and has unit mass. So,
\begin{equation}
    \bar F(z)
    =
    \int_{-\infty}^{\infty}
    J'(z+au)\kappa(u)\,\mathrm{d}u,
    \label{eq:Fbar-kernel}
\end{equation}
and $\bar F(z)$ is exactly a convexly weighted average of the
nearby derivatives sampled over $[z-a,z+a]$.

\begin{figure}[!t]
    \centering
    \includegraphics[width=\linewidth]
    {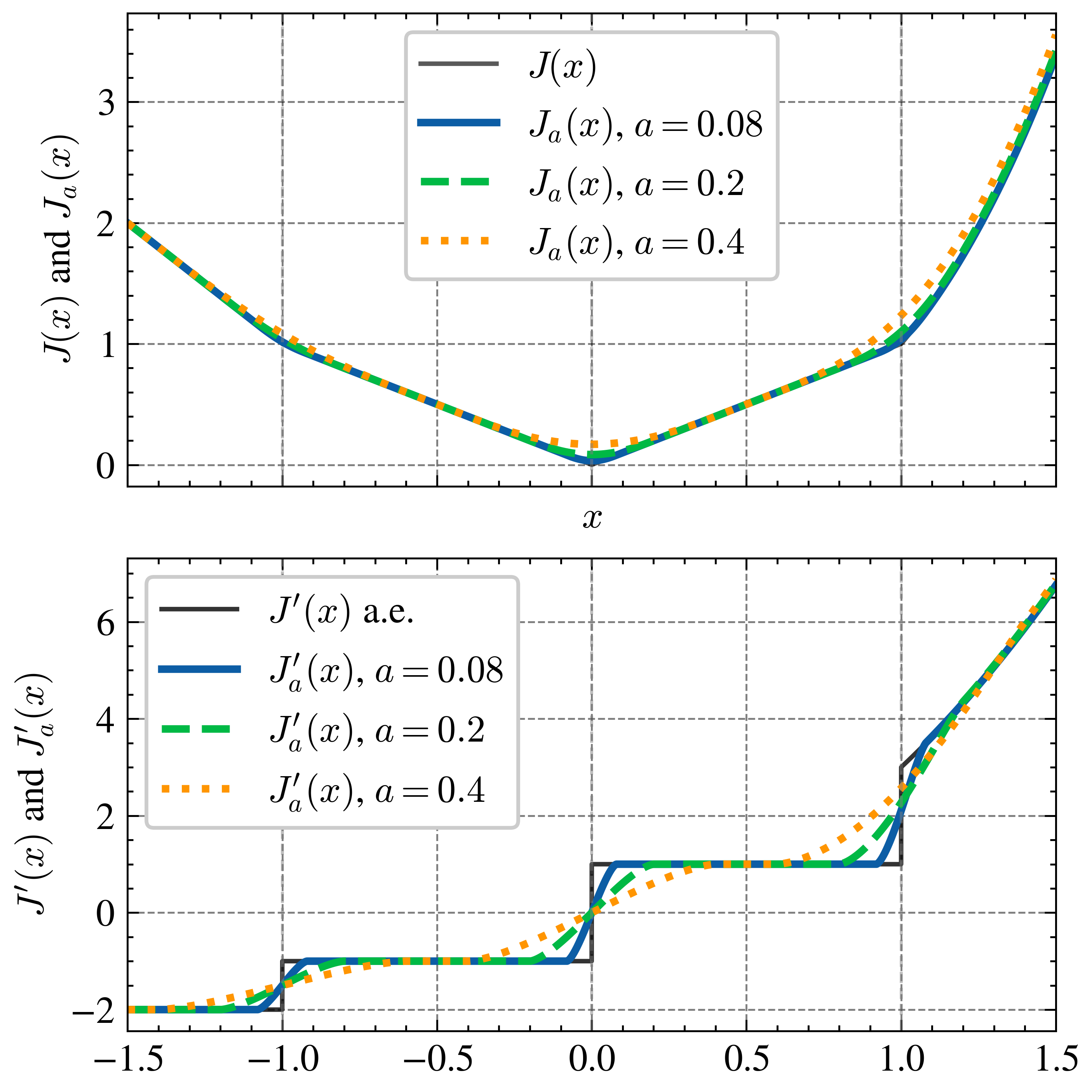}
    \caption{Semicircle smoothing of the nonsmooth objective
    $J(x)=\max\{-x,x,-2x-1,x^3\}$. The top panel shows $J$ and the corresponding
    smoothed objectives $J_a=\kappa_a*J$ for several perturbation amplitudes. The bottom panel shows the almost-everywhere derivative of \(J\) and the averaged slopes \(J_a'=\kappa_a*J'\).}
    \label{fig:averaging_example_simple}
\end{figure}

There is another equivalent interpretation. Define the
smoothed objective
\begin{equation}
    J_a(x)
    :=
    \int_{-\infty}^{\infty}
    J(x+au)\kappa(u)\,\mathrm{d}u.
    \label{eq:motivating-smoothed-objective}
\end{equation}
Since \(J\) is locally Lipschitz,
Proposition~\ref{prop:diff-under-integral} in
Appendix~\ref{app:diff-under-integral} shows that \(J_a\) is
differentiable and that
\begin{equation}
    J_a'(x)
    =
    \int_{-\infty}^{\infty}
    J'(x+au)\kappa(u)\,\mathrm{d}u.
    \label{eq:motivating-Ja-derivative}
\end{equation}
Comparing \eqref{eq:Fbar-kernel} and
\eqref{eq:motivating-Ja-derivative} shows that indeed
\begin{equation}
    \bar F(z)=J_a'(z).
    \label{eq:motivating-average-is-gradient}
\end{equation}
So the averaged ES system is exactly the gradient flow of the
smoothed objective \(J_a\), rather than a small-amplitude approximation of gradient descent on \(J\). Averaging theory \cite[Theorem~4.3.6]{sanders2007averaging} tells us that the trajectories of 
\begin{equation}
    \frac{\mathrm{d}z}{\mathrm{d}\tau}
    =-\varepsilon J_a'(z),
    \label{eq:motivating-gradient-flows}
\end{equation}
are close to the trajectories of \eqref{eq:ES-classical-tau} on finite time intervals.

The average dynamics also have the interpretation of a convolution. Take
$y=au$ and define the scaled kernel
\begin{equation}
    \kappa_a(y)
    :=
    \frac{1}{a}\kappa\!\left(\frac{y}{a}\right)
    =
    \begin{cases}
        \dfrac{2}{\pi a^2}\sqrt{a^2-y^2}, & |y|\leq a,\\[1.2ex]
        0, & |y|>a.
    \end{cases}
    \label{eq:motivating-kappa-a}
\end{equation}
Using the convention
\[
    (f*g)(x)
    :=
    \int_{-\infty}^{\infty} f(y)g(x-y)\,\mathrm{d}y,
\]
and the evenness of $\kappa_a$, the functions $J_a$ and $J_a'$ become
\begin{equation}
    J_a=\kappa_a*J,
    \qquad
    \bar F=J_a'=\kappa_a*J'.
    \label{eq:motivating-convolution-identities}
\end{equation}
Fig.~\ref{fig:averaging_example_simple} illustrates these two
interpretations for
\[
    J(x)=\max\{-x,x,-2x-1,x^3\}.
\]
The objective is nonsmooth at $x=-1$, $x=0$, and $x=1$. Its
almost-everywhere derivative therefore has jump discontinuities, whereas
the averaged slopes $J_a'=\kappa_a*J'$ are continuous. Equivalently, the
smoothed objectives $J_a=\kappa_a*J$ round the kinks of $J$. Increasing
$a$ widens the support of $\kappa_a$, so both $J_a(x)$ and $J_a'(x)$
incorporate information from a larger neighborhood of $x$.

Extending this interpretation to multiple dimensions requires addressing two issues. First, a periodic perturbation generally traces only a closed one-dimensional curve in $n>1$ and does not explore the full neighborhood of the parameter estimate. Second, the classical demodulator produces component-dependent smoothing kernels. The proposed design uses rationally independent frequencies for dense exploration and a matched demodulator that produces the same kernel in every component.

\section{Preliminaries}
\label{sec:preliminaries}

\subsection{Kronecker--Weyl Theorem}
\begin{theorem}[Kronecker--Weyl]
\label{thm:equidistribution-independent-phases}
Let $\omega_i$ for $i =1, \dots, r$ be \emph{rationally independent}, namely they satisfy
\begin{equation}
    \ell^\top\omega\neq 0
    \qquad
    \text{for every }\ell\in\mathbb{Q}^r\setminus\{0\}.
\end{equation}
If $f:\mathbb{R}^r\to\mathbb{R}$ is continuous and $2\pi$-periodic
in each argument, then
\begin{equation}
\label{eq:independent-phase-average}
\begin{aligned}
    &\lim_{T\to\infty}\frac{1}{T}\int_0^T
    f(\omega_1t,\ldots,\omega_rt)\,\mathrm{d}t \\
    &\qquad =
    \frac{1}{(2\pi)^r}\int_{[0,2\pi]^r}
    f(u_1,\ldots,u_r)\,\mathrm{d}u_1\cdots\mathrm{d}u_r .
\end{aligned}
\end{equation}
\end{theorem}
%\vspace{1mm}

This is the rationally independent special case of
\cite[Theorem~2.7]{bailleul2022explicit}. The componentwise periodicity of
\(f\) allows each argument to be taken modulo \(2\pi\). The normalized Haar measure appearing in that theorem is represented in
these phase coordinates by the normalized Lebesgue measure
\((2\pi)^{-r}\,\mathrm{d}u_1\cdots\mathrm{d}u_r\) on
\([0,2\pi]^r\), giving \eqref{eq:independent-phase-average}.

Informally, rational independence prevents the relative phases from becoming
locked into a repeating lower-dimensional pattern. Over long times, the
phase trajectory \(\omega t\), with each component taken modulo \(2\pi\),
explores \([0,2\pi]^r\) uniformly, so the time average along a single signal
equals the uniform average over all phase variables.

\subsection{General Averaging}
\label{subsec:general-averaging}

Consider the initial-value problem
\begin{equation}
    \frac{\mathrm{d}x}{\mathrm{d}\tau}
    =
    \varepsilon f(x,\tau),
    \qquad
    x(\tau_0)=x_0,
    \label{eq:general-averaging-original}
\end{equation}
where \(\tau_0\in\mathbb{R}\). When it exists independently of
\(\tau_0\), define the common long-time average of \(f\) by
\begin{equation}
    \bar f(x)
    :=
    \lim_{T\to\infty}
    \frac{1}{T}
    \int_{\tau_0}^{\tau_0+T}
    f(x,s)\,\mathrm{d}s.
    \label{eq:general-average-limit}
\end{equation}
The corresponding average system is
\begin{equation}
    \frac{\mathrm{d}z}{\mathrm{d}\tau}
    =
    \varepsilon\bar f(z),
    \qquad
    z(\tau_0)=x_0,
    \label{eq:general-averaging-average}
\end{equation}
where \(x(\tau),z(\tau),x_0\in D\subset\mathbb{R}^n\),
\(\tau\in[\tau_0,\infty)\), and
\(\varepsilon\in(0,\varepsilon_0]\). The following result summarizes
the portion of general averaging theory used in this paper.

\begin{theorem}[General Averaging]
\label{thm:Sanders-general-averaging}
Suppose that \(f\) is continuous and locally Lipschitz in \(x\),
uniformly in \(\tau\) on compact subsets of \(D\), and that the limit
in \eqref{eq:general-average-limit} exists uniformly with respect to
\(x\) on compact subsets of \(D\) and
\(\tau_0\in\mathbb{R}\). Fix \(L>0\), and suppose that the solution
\(z(\tau)\) of \eqref{eq:general-averaging-average} remains in an
interior compact subset of \(D\) for
\(
    \tau_0
    \leq
    \tau
    \leq
    \tau_0+L/\varepsilon.
\)
Then there exists an associated order function
\(\delta_1(\varepsilon)\) satisfying
\begin{equation}
    \lim_{\varepsilon\to0}\delta_1(\varepsilon)=0
    \label{eq:Sanders-delta-limit}
\end{equation}
such that
\begin{equation}
    \|x(\tau)-z(\tau)\|
    =
    O\!\left(\sqrt{\delta_1(\varepsilon)}\right),
    \qquad
    \tau_0
    \leq
    \tau
    \leq
    \tau_0+\frac{L}{\varepsilon},
    \label{eq:Sanders-general-error}
\end{equation}
as \(\varepsilon\to0\), uniformly over the time interval
and with respect to \(\tau_0\in\mathbb{R}\) and \(x_0\) on compact
subsets of \(D\).
\end{theorem}
%\vspace{1mm}

In the terminology of
\cite[Definition~4.2.4]{sanders2007averaging}, the continuity,
Lipschitz, and uniform-average hypotheses above state that \(f\) is a
``KBM-vector field.'' Lemma~4.3.4 of
\cite{sanders2007averaging} establishes
\eqref{eq:Sanders-delta-limit}, while
\cite[Theorem~4.3.6]{sanders2007averaging} gives
\eqref{eq:Sanders-general-error}. If \(f\) depends on additional
parameters, the parameters and initial conditions are taken
independently of \(\varepsilon\), and the limit in
\eqref{eq:general-average-limit} is assumed to be uniform in those
parameters; see the discussion below
\cite[Definition~4.2.4]{sanders2007averaging}. Since \(L\) is fixed
but arbitrary, the theorem gives trajectory convergence on every
fixed finite interval in slow time \(s=\varepsilon \tau\).

The original result
\cite[Theorem~4.3.6]{sanders2007averaging} is stated with initial time
zero, and the form above is its immediate time-translation extension.
To see this, set \(r=\tau-\tau_0\) and define
\[
    f_{\tau_0}(x,r):=f(x,r+\tau_0).
\]
The translated initial-value problem begins at \(r=0\), with
\(\tau_0\) appearing as an additional parameter independent of
\(\varepsilon\). Moreover,
\[
    \frac{1}{T}\int_0^T
    f_{\tau_0}(x,r)\,\mathrm{d}r
    =
    \frac{1}{T}\int_{\tau_0}^{\tau_0+T}
    f(x,s)\,\mathrm{d}s.
\]
Thus the assumed uniformity in \(\tau_0\) is precisely the parameter
uniformity required in the KBM definition
\cite[Definition~4.2.4]{sanders2007averaging}. Applying
\cite[Theorem~4.3.6]{sanders2007averaging} to the resulting
initial-value problem with initial time \(r=0\), uniformly with respect
to \(\tau_0\), and returning to \(\tau=r+\tau_0\) gives
\eqref{eq:Sanders-general-error}.

\subsection{Practical Global Uniform Asymptotic Stability} \label{subsec:PGUAS}

Consider the parameter-dependent system
\begin{equation}
    \dot{x}
    =
    f(t,x,\varepsilon),
    \qquad
    x(t_0)=x_0,
    \qquad
    0<\varepsilon\leq\varepsilon_0,
    \label{eq:practical-parameter-system}
\end{equation}
and the limiting system
\begin{equation}
    \dot{z}
    =
    g(t,z),
    \qquad
    z(t_0)=x_0.
    \label{eq:practical-limit-system}
\end{equation}

\begin{definition} \label{def:GUAS}
    The origin of \eqref{eq:practical-limit-system} is said to be
\emph{globally uniformly asymptotically stable} (GUAS) if there exists \(\beta\in\mathcal{KL}\) such that every
solution is defined for all \(t\geq t_0\) and satisfies
\begin{equation}
    \|z(t)\|
    \leq
    \beta\bigl(\|x_0\|,t-t_0\bigr),
    \qquad
    t\geq t_0.
    \label{eq:GUAS-KL}
\end{equation}
\end{definition}
%\vspace{1mm}

\begin{definition} \label{def:PGUAS}
    The origin of \eqref{eq:practical-parameter-system} is said to be
\emph{practically globally uniformly asymptotically stable} (PGUAS)
if there exists a function \(\beta\in\mathcal{KL}\) such that,
for every \(\Delta,\nu>0\), there exists
\(\varepsilon^*(\Delta,\nu)\in(0,\varepsilon_0]\) such that, for every
\(t_0\in\mathbb{R}\), \(\|x_0\|\leq\Delta\), and
\(0<\varepsilon<\varepsilon^*(\Delta,\nu)\), the corresponding
solution is defined for all \(t\geq t_0\) and satisfies
\begin{equation}
    \|x(t,\varepsilon)\|
    \leq
    \beta\bigl(\|x_0\|,t-t_0\bigr)+\nu,
    \qquad
    t\geq t_0,
    \label{eq:PGUAS-KL}
\end{equation}
where the function \(\beta\) is independent of
\(\Delta\), \(\nu\), and \(\varepsilon\).
\end{definition}
%\vspace{1mm}

We use Definitions~\ref{def:GUAS} and~\ref{def:PGUAS} as the meanings
of GUAS and PGUAS throughout this paper. Moreau and Aeyels in \cite{moreau2000practical} instead
formulate these notions through separate uniform stability,
boundedness, and attractivity properties.\footnote{The two
formulations can be shown to be equivalent, which justifies using the
same terminology. Their equivalence is not needed here:
Proposition~\ref{prop:moreau-KL-bound} derives the
\(\mathcal{KL}\) conclusion in Definition~\ref{def:PGUAS} directly
from \cite[Theorem~1]{moreau2000practical}.} The
following theorem is a consequence of
\cite[Theorem~1]{moreau2000practical}, given in terms of $\mathcal{KL}$ functions. Proposition~\ref{prop:moreau-KL-bound} in Appendix~\ref{app:KL-bounds} includes the argument yielding the bound in Definition~\ref{def:PGUAS}.

\begin{theorem}[Practical Stability]
\label{thm:practical-stability}
Consider systems
\eqref{eq:practical-parameter-system} and
\eqref{eq:practical-limit-system}. Suppose the following conditions
hold:
\begin{enumerate}
    \item (\emph{Existence and Uniqueness})
    For each \(\varepsilon\in(0,\varepsilon_0]\), the function
    \(f(\cdot,\cdot,\varepsilon)\) is continuous and locally Lipschitz
    in \(x\), uniformly in \(t\) on compact time intervals. The
    function \(g\) is continuous and locally Lipschitz in \(z\),
    uniformly in \(t\) on compact time intervals.

    \item (\emph{Convergence of Trajectories})
For every \(T>0\), every compact set \(K\subset\mathbb{R}^{n}\) such that
the solution \(z(t)\) is defined on \([t_0,t_0+T]\) for every
\(t_0\in\mathbb{R}\) and \(x_0\in K\), and every \(d>0\), there exists
\(\bar{\varepsilon}\in(0,\varepsilon_0]\) such that, for every
\(t_0\in\mathbb{R}\), \(x_0\in K\), and
\(0<\varepsilon<\bar{\varepsilon}\), the solution
\(x(t,\varepsilon)\) exists on \([t_0,t_0+T]\) and
\begin{equation}
    \|x(t,\varepsilon)-z(t)\|
    <
    d,
    \qquad
    t_0\leq t\leq t_0+T.
    \label{eq:uniform-trajectory-convergence}
\end{equation}
\end{enumerate}
If the origin is a GUAS equilibrium of
\eqref{eq:practical-limit-system}, then the origin of
\eqref{eq:practical-parameter-system} is PGUAS.
\end{theorem}
%\vspace{1mm}

The implication of Theorem~\ref{thm:practical-stability} is that a parameter-dependent system inherits practical global uniform asymptotic stability whenever its trajectories approximate, uniformly over finite time intervals, those of a GUAS limiting system. Because GUAS implies forward completeness of the limiting system, the existence qualification on \(z\) in Condition~2 is automatic in our application. 

\section{Extremum Seeking Design for Locally Lipschitz Objectives}
\label{sec:design}

We consider the minimization of a static objective
$J:\mathbb{R}^n\to\mathbb{R}$ using only evaluations of $J$. The parameter
estimate is denoted by $\hat x\in\mathbb{R}^n$.

\begin{assumption}[Objective Lipschitzness]
\label{ass:objective-lipschitz}
The objective $J:\mathbb{R}^n\to\mathbb{R}$ is locally Lipschitz.
\end{assumption}
%\vspace{1mm}

No derivative of $J$ is required by the algorithm. Assumption
\ref{ass:objective-lipschitz} guarantees local existence and uniqueness of
the resulting dynamics. It also implies, by Rademacher's theorem, that
$\nabla J$ exists almost everywhere and is essentially bounded on compact
sets.

Let $a,k,\omega>0$, and choose frequencies
$\hat\omega_1,\ldots,\hat\omega_n>0$ satisfying the following condition.

\begin{assumption}[Rationally Independent Frequencies]
\label{ass:rationally-independent-frequencies}
The relative-frequency vector
\(\hat\omega:=(\hat\omega_1,\ldots,\hat\omega_n)^\top\)
is rationally independent; that is,
\begin{equation}
    \ell^\top\hat\omega\neq 0
    \qquad
    \text{for every } \ell\in\mathbb{Q}^n\setminus\{0\}.
    \label{eq:rational-independence}
\end{equation}
\end{assumption}
%\vspace{1mm}

This is the key frequency requirement for applying Theorem~\ref{thm:equidistribution-independent-phases}. Together with the general
averaging theorem, it will produce a particularly clean expression for the
average dynamics. A straightforward choice is obtained by choosing distinct
prime numbers $p_1,\ldots,p_n$ and setting
\begin{equation}
    \hat\omega_i=\sqrt{p_i} \, .
    \label{eq:explicit-frequency-choice}
\end{equation}

Define the
perturbation and demodulation signals by
\begin{align}
    S_i(t)
    &:={}
    a\sin(\omega\hat\omega_i t),
    \label{eq:design-perturbation}
    \\
    M_i(t)
    &:={}
    \frac{2^n}{a}
    \sin(\omega\hat\omega_i t)
    \prod_{j\neq i}
    \cos^2(\omega\hat\omega_j t).
    \label{eq:design-demodulator}
\end{align}
The proposed
extremum seeking law is the $n$-dimensional system
\begin{equation}
    \dot{\hat x}(t)
    =
    -kJ\bigl(\hat x(t)+S(t)\bigr)M(t),
    \qquad
    \hat x(t_0)=x_0,
    \label{eq:proposed-ES-vector}
\end{equation}
where \(t_0\in\mathbb{R}\) and \(x_0\in\mathbb{R}^n\).

There is no change in the main result if a constant initial phase is included
by replacing each $\omega\hat\omega_i t$ with
$\omega\hat\omega_i t+\phi_i$. A vector amplitude may also be used. In that
case, $S(t)=a\odot\sin(\omega\hat\omega t)$, the factor $1/a$ in the $i$th
demodulation channel is replaced by $1/a_i$. Throughout the analysis, trigonometric functions of vectors are understood componentwise.

The averaging accuracy depends on making the ratio $k/\omega$ small. A practitioner may therefore decrease $k$, increase $\omega$, or tune both. The gain $k$ controls the adaptation time scale, while $\omega$ controls the perturbation time scale. The parameter $a$ determines the exploration region around the parameter estimate.

\section{Averaging Analysis}
\label{sec:analysis}
This section derives the averaged system in four steps. First, we introduce
the fast time \(\tau=\omega t\). Second, the Kronecker--Weyl theorem converts
the long-time average into an integral over the phase variables.
Third, integration by parts reveals the almost-everywhere gradient of \(J\).
Finally, a spatial change of variables identifies a common smoothing kernel
and shows that the average system is the gradient flow of the smoothed
objective \(J_a\).

\subsection{General Averaging Form}
\label{subsec:ES-general-averaging}

Starting from the $n$-dimensional dynamics
\eqref{eq:proposed-ES-vector}, introduce fast time
\begin{equation}
    \tau:=\omega t,
    \qquad
    \varepsilon:=\frac{k}{\omega}.
    \label{eq:fast-time-and-epsilon}
\end{equation}
This places the system in the standard form for averaging and collects the adaptation gain $k$ and perturbation scale $\omega$ into the single parameter
$\varepsilon$. So, the small parameter required by Theorem~\ref{thm:Sanders-general-averaging} can be produced by $k$ small, $\omega$ large, or $k/\omega$ small.

In fast time,
\begin{align}
    S_i(\tau/\omega)
    &=a\sin(\hat \omega_i \tau),
    \label{eq:fast-time-perturbation}
    \\
    M_i(\tau/\omega)
    &=
    \frac{2^n}{a}
    \sin(\hat\omega_i\tau)
    \prod_{j\neq i}\cos^2(\hat\omega_j\tau).
    \label{eq:fast-time-demodulator}
\end{align}
For $i=1,\ldots,n$, define
\begin{equation}
    F_i(x,\theta)
    :=
    -\frac{2^n}{a}
    J\bigl(x+a\sin\theta\bigr)
    \sin\theta_i \,
    \prod_{j\neq i}\cos^2\theta_j\, ,
    \label{eq:phase-vector-field}
\end{equation}
then \eqref{eq:proposed-ES-vector} becomes
\begin{equation}
    \frac{\mathrm{d}\hat x}{\mathrm{d} \tau}
    =
    \varepsilon F(\hat x,\hat\omega\tau).
    \label{eq:ES-fast-time}
\end{equation}
On every compact subset of $\mathbb{R}^n$, the right-hand side is Lipschitz
in $\hat x$ uniformly in $\tau$. The regularity hypotheses of the general averaging theorem are
therefore satisfied. It remains to show that the moving-window average
exists uniformly with respect to the state on compact sets and the
initial time.

\subsection{Kronecker--Weyl Average}
\label{subsec:ES-KW-average}

For each component of \eqref{eq:ES-fast-time}, the limit
required by the general averaging theorem is
\begin{equation}
    \bar F_i(z)
    :=
    \lim_{T\to\infty}
    \frac{1}{T}
    \int_{\tau_0}^{\tau_0+T}
    F_i(z,\hat\omega\tau)\,\mathrm{d}\tau.
    \label{eq:componentwise-average-limit}
\end{equation}
The corresponding average system is
\begin{equation}
    \frac{\mathrm{d}z}{\mathrm{d}\tau}
    =
    \varepsilon\bar F(z),
    \qquad
    z(\tau_0)=\hat x(\tau_0).
    \label{eq:ES-average-system}
\end{equation}
For any fixed \(\tau_0\), the change of variables
\(r=\tau-\tau_0\) gives
\[
    \frac{1}{T}
    \int_{\tau_0}^{\tau_0+T}
    F_i(z,\hat\omega\tau)\,\mathrm{d}\tau
    =
    \frac{1}{T}
    \int_0^T
    F_i\bigl(z,\hat\omega r+\hat\omega\tau_0\bigr)
    \,\mathrm{d}r.
\]
The integrand is continuous and \(2\pi\)-periodic in each of its phase
arguments. Assumption
\ref{ass:rationally-independent-frequencies} therefore permits
application of
Theorem~\ref{thm:equidistribution-independent-phases} to the
periodic function, shifted in phase by a constant $\hat\omega\tau_0$. Translation invariance of the integral gives
\begin{multline}
    \bar F_i(z)
    =
    -\frac{1}{a\pi^n}
    \int_{[0,2\pi]^n}
    J\bigl(z+a\sin\theta\bigr)
    \sin\theta_i
    \prod_{j\neq i}\cos^2\theta_j
    \,\mathrm{d}\theta.
    \label{eq:ES-phase-average}
\end{multline}
Here, \(\theta\) replaces the phase vector in the arguments of \(F\),
and \(\mathrm{d}\theta\) is understood as
\(\mathrm{d}\theta_1\cdots\mathrm{d}\theta_n\). The factor \(2^n\) in
\eqref{eq:phase-vector-field} cancels against the factor
\((2\pi)^n\).

Fig.~\ref{fig:rational-independent-perturbation-trajectories} illustrates the
role of Assumption~\ref{ass:rationally-independent-frequencies} for $n=2$. Rationally dependent frequencies produce a closed periodic
trajectory that explores only a one-dimensional subset of the square. By
contrast, rationally independent relative frequencies produce a trajectory which is dense in
\([-1,1]^2\).

\begin{figure}[!t]
    \centering
    \hspace*{-0.2cm}\includegraphics[width=\linewidth]
    {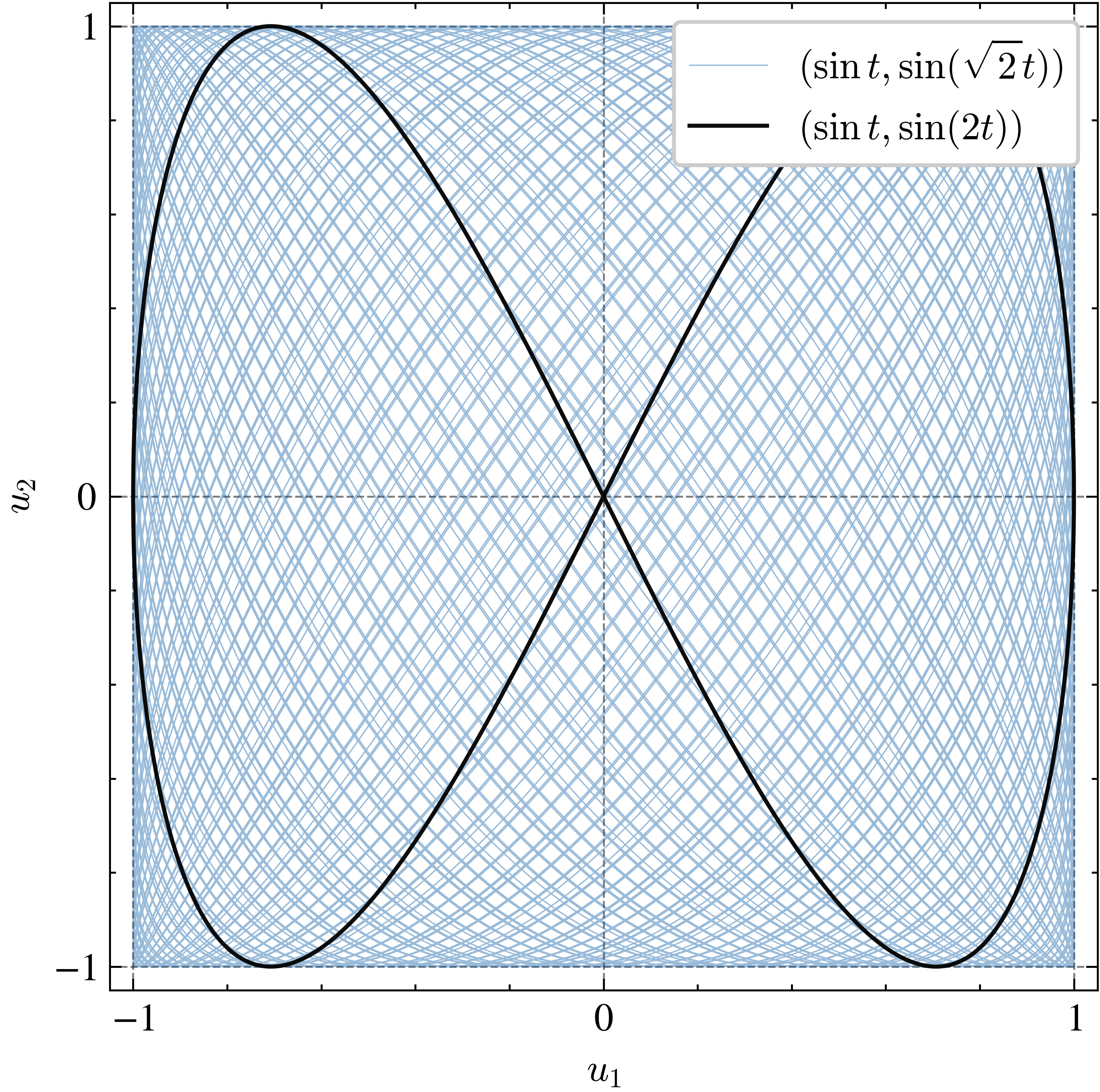}
    \caption{Spatial trajectories generated by two-dimensional sinusoidal
    perturbations. The black curve corresponds to the relative
    frequencies \(\hat\omega=(1,2)\) and forms a closed periodic orbit. The
    blue curve shows a finite time segment corresponding to the rationally
    independent relative frequencies
    \(\hat\omega=(1,\sqrt{2})\); its infinite time image is dense in
    \([-1,1]^2\).}
    \label{fig:rational-independent-perturbation-trajectories}
\end{figure}

For each fixed \(z\) and \(\tau_0\),
Theorem~\ref{thm:equidistribution-independent-phases} guarantees that
the limit in \eqref{eq:componentwise-average-limit} exists. The
following corollary upgrades this pointwise convergence to convergence
uniform in \(z\) on compact sets, the initial time, and any constant
phase offset.

\begin{corollary}[Uniform Kronecker--Weyl Averaging]
\label{cor:uniform-KW-average}
Let \(H:\mathbb{R}^m\times\mathbb{R}^r\to\mathbb{R}^q\) be continuous
and \(2\pi\)-periodic in each component of its second argument. Suppose
that \(\xi\in\mathbb{R}^r\) satisfies
\(
    \ell^\top\xi\neq0
\)
for every
\(
    \ell\in\mathbb{Q}^r\setminus\{0\},
\)
and define
\[
    \bar H(x)
    :=
    \frac{1}{(2\pi)^r}
    \int_{[0,2\pi]^r}H(x,u)\,\mathrm{d}u.
\]
Then, for every compact \(K\subset\mathbb{R}^m\),
\begin{equation}
    \frac{1}{T}
    \int_{\tau_0}^{\tau_0+T}
    H(x,\xi s+\varphi)\,\mathrm{d}s
    \longrightarrow
    \bar H(x)
    \quad\text{as }T\to\infty,
    \label{eq:uniform-KW-average}
\end{equation}
uniformly with respect to
\(x\in K\), \(\varphi\in[0,2\pi]^r\), and
\(\tau_0\in\mathbb{R}\).
\end{corollary}
%\vspace{1mm}

Applied with \(H=F\) and \(\xi=\hat\omega\),
Corollary~\ref{cor:uniform-KW-average} verifies the uniformity of the limit in \eqref{eq:general-average-limit} on compact state sets,
uniformly with respect to the initial time and any constant phase
offset, as required by Theorem~\ref{thm:Sanders-general-averaging}.

Under the fast-time transformation
\(\tau=\omega t\), the original initial time \(t_0\) corresponds to
\(\tau_0=\omega t_0\). Since the trajectory-closeness estimate is
uniform in \(\tau_0\), it may be evaluated at this value and then
returned to the original time variable. This gives the initial-time uniformity required
by Theorem~\ref{thm:practical-stability}. The uniformity in
\(\varphi\) also covers the fixed perturbation phases allowed in
Section~\ref{sec:design}, replacing each $\omega\hat\omega_i t$ with
$\omega\hat\omega_i t+\phi_i$.

\subsection{Integration by Parts}

We next evaluate \eqref{eq:ES-phase-average}. Fixing the other variables, the map
\[
    \theta_i
    \mapsto
    J\bigl(z+a\sin\theta\bigr)
\]
is Lipschitz, hence absolutely continuous, with derivative
\[
    a\,\partial_iJ\bigl(z+a\sin\theta\bigr)\cos\theta_i
\]
almost everywhere. Thus, the integration-by-parts formula for absolutely
continuous functions
\cite[Ch.~3, Exercise~35]{folland1999real} applies. Taking
\(U=J(z+a\sin\theta)\) and
\(\mathrm{d}V=\sin\theta_i\,\mathrm{d}\theta_i\) gives
\begin{multline}
    \int_0^{2\pi}
    J\bigl(z+a\sin\theta\bigr)
    \sin\theta_i\,\mathrm{d}\theta_i
    ={}\\
    a\int_0^{2\pi}
    \partial_iJ\bigl(z+a\sin\theta\bigr)
    \cos^2\theta_i\,\mathrm{d}\theta_i.
    \label{eq:phase-IBP}
\end{multline}
The boundary term is zero. By Fubini's theorem, we may apply
\eqref{eq:phase-IBP} to the inner \(\theta_i\)-integral in
\eqref{eq:ES-phase-average}. The factor \(a\) then cancels, giving
\begin{equation}
    \bar F_i(z)
    =
    -\frac{1}{\pi^n}
    \int_{[0,2\pi]^n}
    \partial_iJ\bigl(z+a\sin\theta\bigr)
    \prod_{j=1}^n\cos^2\theta_j
    \,\mathrm{d}\theta,
    \label{eq:average-after-IBP}
\end{equation}
where \(\partial_iJ\) is understood almost everywhere.

Crucially, the same term
\(\pi^{-n}\prod_{j=1}^n\cos^2\theta_j\geq0\) appears in every component, and the averaged field can already be considered a convexly weighted average of nearby gradients of $-J$.

\subsection{Spatial Coordinate Transformation}
We now make the coordinate transformation
\[u_i=\sin\theta_i, \] to reexpress the integral \eqref{eq:average-after-IBP} on
$u\in[-1,1]^n$. The one-dimensional identity underlying the transformation
is
\begin{equation}
    \int_0^{2\pi}
    h(\sin\theta)\cos^2\theta\,\mathrm{d}\theta
    =
    2\int_{-1}^{1}
    h(u)\sqrt{1-u^2}\,\mathrm{d}u
    \label{eq:one-dimensional-semicircle-change}
\end{equation}
for every integrable $h:[-1,1]\to\mathbb{R}$. To see this, periodicity
allows the left-hand integral to be taken over
$[-\pi/2,3\pi/2]$. Splitting this interval at $\pi/2$ produces one branch on
which sine increases from $-1$ to $1$ and one branch on which it decreases
from $1$ to $-1$. Applying $u=\sin\theta$ on the two branches gives the same
integral, producing the factor $2$ in
\eqref{eq:one-dimensional-semicircle-change}.

Applying this identity successively, for $2^n$ branches, to
$\theta_1,\ldots,\theta_n$ gives, for every
integrable $H:[-1,1]^n\to\mathbb{R}$, the $n$-dimensional identity:
\begin{multline}
    \int_{[0,2\pi]^n}
    H(\sin\theta)
    \prod_{j=1}^n\cos^2\theta_j
    \,\mathrm{d}\theta
    ={}\\
    2^n\int_{[-1,1]^n}
    H(u)
    \prod_{j=1}^n\sqrt{1-u_j^2}
    \,\mathrm{d}u,
    \label{eq:multidimensional-semicircle-change}
\end{multline}
where $\mathrm{d}u$ denotes
$\mathrm{d}u_1\cdots\mathrm{d}u_n$. 

Applying
\eqref{eq:multidimensional-semicircle-change} componentwise to
\eqref{eq:average-after-IBP} yields
\begin{equation}
    \bar F(z)
    =
    -\int_{[-1,1]^n}
    \nabla J(z+au)\kappa(u)\,\mathrm{d}u,
    \label{eq:average-common-kernel-u}
\end{equation}
where
\begin{equation}
    \kappa(u)
    :=
    \left(\frac{2}{\pi}\right)^n
    \prod_{j=1}^n\sqrt{1-u_j^2}.
    \label{eq:product-semicircle-kernel}
\end{equation}
The kernel is nonnegative and has unit mass because
\begin{equation}
    \int_{-1}^1\sqrt{1-v^2}\,\mathrm{d}v
    =
    \frac{\pi}{2}.
    \label{eq:semicircle-unit-mass}
\end{equation}
Thus, \(\bar F\) is exactly a convexly weighted average of the
almost-everywhere descent directions \(-\nabla J\) throughout
\(z+a[-1,1]^n\).

Now we show that the interpretation is equivalent to the gradient flow of a smoothed version of the objective function. Define the smoothed objective
\begin{equation}
    J_a(x)
    :=
    \int_{[-1,1]^n}
    J(x+au)\kappa(u)\,\mathrm{d}u.
    \label{eq:smoothed-objective}
\end{equation}
By Proposition~\ref{prop:diff-under-integral} in Appendix~\ref{app:diff-under-integral}, $J_a$ is differentiable,
\begin{equation}
    \nabla J_a(x)
    =
    \int_{[-1,1]^n}
    \nabla J(x+au)\kappa(u)\,\mathrm{d}u.
    \label{eq:gradient-smoothed-objective}
\end{equation}
Comparing \eqref{eq:average-common-kernel-u} and
\eqref{eq:gradient-smoothed-objective} gives
\begin{equation}
    \bar F(z)=-\nabla J_a(z).
    \label{eq:average-is-gradient}
\end{equation}
Moreover, \(\nabla J_a=-\bar F\) is locally Lipschitz by
\eqref{eq:ES-phase-average} and the local Lipschitz continuity of \(J\).

The average dynamics in fast time with $\tau = \omega t$ are
\begin{equation}
    \frac{\mathrm{d}z}{\mathrm{d}\tau}
    =
    -\varepsilon\nabla J_a(z).
    \label{eq:averaged-gradient-flow-fast-time}
\end{equation}
Equivalently, in slow time
\[
    s:=\varepsilon\tau=kt,
    \qquad
    s_0:=\varepsilon\tau_0=kt_0,
\]
the average system is the gradient flow
\begin{equation}
    \frac{\mathrm{d}z}{\mathrm{d}s}
    =
    -\nabla J_a(z).
    \label{eq:averaged-gradient-flow-slow-time}
\end{equation}

\section{Main Result}
\label{sec:main-result}

The stability requirement is now imposed directly on the average dynamics
identified above.

\begin{assumption}[GUAS of Average Dynamics]
\label{ass:average-system-GUAS}
For the selected amplitude \(a>0\), \(x_a\in\mathbb{R}^n\) is a GUAS equilibrium of \eqref{eq:averaged-gradient-flow-slow-time}.
\end{assumption}
%\vspace{1mm}

Assumption~\ref{ass:average-system-GUAS} is deliberately stated at the level
needed for applying Theorem~\ref{thm:practical-stability}. In particular, it does not require
$J_a$ to be coercive. It permits, for example, a well-shaped objective that
approaches a finite value and becomes flat at infinity, provided its gradient
flow is nevertheless GUAS. It also allows undesired stationary points of the
original objective to be removed by averaging at the selected amplitude $a$. Several familiar conditions imply Assumption
\ref{ass:average-system-GUAS}. If $J$ is strongly convex, then the weighted
average $J_a$ is strongly convex, and its gradient flow is GUAS at its unique
minimizer. More generally, the outward-pointing condition
\begin{equation}
    (x-x_a)^\top\nabla J_a(x)
    \geq
    \alpha\bigl(\|x-x_a\|\bigr),
    \label{eq:average-outward-pointing}
\end{equation}
where $\alpha$ is continuous and positive definite, proves GUAS using
$\|x-x_a\|^2$ as a Lyapunov function. Strong convexity with parameter
$\mu>0$ implies \eqref{eq:average-outward-pointing} with
$\alpha(r)=\mu r^2$.

Corollary~\ref{cor:uniform-KW-average} and
Theorem~\ref{thm:Sanders-general-averaging} imply that, for every compact
set of initial conditions and every \(L,d>0\),
\begin{equation}
    \|\hat x(\tau)-z(\tau)\|<d,
    \qquad
    \tau_0\leq\tau\leq\tau_0+\frac{L}{\varepsilon},
    \label{eq:ES-fast-time-closeness}
\end{equation}
for all sufficiently small \(\varepsilon\), uniformly with respect to
\(\tau_0\). In slow time \(s=\varepsilon\tau\), this is precisely the
trajectory-convergence condition of
Theorem~\ref{thm:practical-stability}. The local Lipschitz properties
established above satisfy the remaining regularity conditions, yielding
the following result under
Assumption~\ref{ass:average-system-GUAS}.

\begin{theorem}[Main Result]
\label{thm:main-ES-result}
Suppose Assumptions~\ref{ass:objective-lipschitz},
\ref{ass:rationally-independent-frequencies}, and
\ref{ass:average-system-GUAS} hold. Then $x_a$ is PGUAS with respect to $\varepsilon = k / \omega$ for \eqref{eq:proposed-ES-vector} in slow time $s = k t$. In particular, there exists
$\beta\in\mathcal{KL}$ such that, for every $\Delta,\nu>0$, there exists
$\varepsilon^*(\Delta,\nu)>0$ for which
\begin{equation}
    0<\varepsilon<\varepsilon^*(\Delta,\nu)
    \label{eq:main-epsilon-condition}
\end{equation}
implies that every solution such that 
$\|\hat x(t_0)-x_a\|\leq\Delta$ satisfies
\begin{equation}
    \|\hat x(t)-x_a\|
    \leq
    \beta\bigl(
    \|\hat x(t_0)-x_a\|,
    k (t-t_0)
    \bigr)
    +\nu, \quad
    t\geq t_0.
    \label{eq:main-SPUAS-bound}
\end{equation}
\end{theorem}
%\vspace{1mm}

\begin{proof}
Let
\[
    s=\varepsilon\tau=kt,
    \qquad
    s_0=\varepsilon\tau_0=kt_0,
\]
and define the slow-time reparameterizations, where $\hat x(t)$ and $z(\tau)$ are the solutions to \eqref{eq:proposed-ES-vector} and \eqref{eq:ES-average-system} respectively,
\[
    \widetilde x(s):=\hat x(s/k),
    \qquad
    \widetilde z(s):=z(s/\varepsilon).
\]
Then
\[
    \frac{\mathrm{d}\widetilde x}{\mathrm{d}s}
    =
    F\left(
        \widetilde x,
        \frac{\hat\omega s}{\varepsilon}
    \right),
    \qquad
    \frac{\mathrm{d}\widetilde z}{\mathrm{d}s}
    =
    -\nabla J_a(\widetilde z),
\]
with
\(
\widetilde x(s_0)=\widetilde z(s_0)=x_0.
\)

Under \(s=\varepsilon\tau\), the fast-time estimate
\eqref{eq:ES-fast-time-closeness} becomes
\[
    \|\widetilde x(s)-\widetilde z(s)\|<d,
    \qquad
    s_0\leq s\leq s_0+L.
\]
For every compact set of initial conditions and every \(L,d>0\), this
estimate holds for all sufficiently small \(\varepsilon\), uniformly
with respect to \(s_0\), because
\eqref{eq:ES-fast-time-closeness} is uniform with respect to
\(\tau_0=s_0/\varepsilon\). Thus, the finite-horizon
trajectory-convergence condition of
Theorem~\ref{thm:practical-stability} holds in slow time.

The existence-and-uniqueness condition of
Theorem~\ref{thm:practical-stability} is satisfied because
\(F(x,\theta)\) is continuous and locally Lipschitz in \(x\), uniformly
in \(\theta\), while \(\nabla J_a\) is locally Lipschitz. The limiting
system is GUAS at \(x_a\) by
Assumption~\ref{ass:average-system-GUAS}. Applying
Theorem~\ref{thm:practical-stability} after translating \(x_a\) to the
origin gives
\[
    \|\widetilde x(s)-x_a\|
    \leq
    \beta\bigl(\|x_0-x_a\|,s-s_0\bigr)+\nu.
\]
Since \(s-s_0=k(t-t_0)\), returning to original time gives
\eqref{eq:main-SPUAS-bound}.
\end{proof}
\begin{remark}
\label{rem:smoothing-bias}
The point \(x_a\) in Theorem~\ref{thm:main-ES-result} is the equilibrium of the gradient flow of the smoothed objective and need not coincide exactly with a minimizer of \(J\). If $J$ is coercive and has a unique minimizer $x^*$, then
$J_a\to J$ locally uniformly as $a\to0$. Whenever Assumption
\ref{ass:average-system-GUAS} holds along the selected sequence of
amplitudes, the corresponding minimizers satisfy $x_a\to x^*$. Thus $a$
controls the bias to the point of convergence, while $\varepsilon=k/\omega$ controls the
averaging error.
\end{remark}

\begin{remark}
There are three things worth mentioning regarding implementation. 1) Square roots of distinct primes provide a simple choice satisfying
Assumption~\ref{ass:rationally-independent-frequencies}. In practice,
the relative frequencies should be chosen to be well separated, since
nearly equal frequencies may delay exploration of the perturbation signal.
Theorem~\ref{thm:main-ES-result} is qualitative in this regard, since \(\varepsilon^*\) may depend on
\(\hat\omega\). 2) Finite-precision hardware replaces
irrational $\hat \omega$ by rational approximations, so the implemented
perturbation is periodic rather than truly dense. Even with exactly rationally
independent frequencies, however, density is an infinite-time property: no perturbation trajectory densely fills the exploration region in finite time.
The finite-time trajectory approximation
\eqref{eq:ES-fast-time-closeness} relies on the finite-time averages appearing
in \eqref{eq:componentwise-average-limit} becoming sufficiently close to the
phase average \eqref{eq:ES-phase-average} over fast-time intervals during
which the parameter estimate evolves slowly. So, rational frequency
approximations with sufficiently good finite-time coverage and long recurrence
periods of $S(t)$ can be expected to produce similar practical averaging behavior.
3) The magnitude of each demodulation channel in
\eqref{eq:design-demodulator} scales as \(2^n/a\), which may amplify
measurement noise in higher dimensions. Since \(k\) multiplies the
demodulated signal, this effect can be moderated by reducing \(k\), at the
cost of slower adaptation; slower practical convergence is also expected
because the perturbation signal must explore a higher-dimensional region.
\end{remark}

\section{Alternative Perturbation and Demodulation Signals}
\label{sec:demodulation-discussion}
\subsection{General Perturbation and Demodulation Signals}
\label{subsec:general-perturbation-demodulation}

We next extend the sinusoidal design to general matched perturbation and
demodulation signals. The signals in
\eqref{eq:design-perturbation}--\eqref{eq:design-demodulator} produce the
product-semicircle kernel \eqref{eq:product-semicircle-kernel}; here, they
are replaced by a general pair designed to produce a prescribed kernel.

Let \(\phi:\mathbb{R}\to[-1,1]\) be Lipschitz and \(2\pi\)-periodic, and
choose \(\hat\omega_1,\ldots,\hat\omega_n\) satisfying
Assumption~\ref{ass:rationally-independent-frequencies}. Define
\begin{equation}
\begin{aligned}
    \theta_i(t)
    &:={}
    \omega\hat\omega_i t,\\
    u_i(t)
    &:={}
    \phi\bigl(\theta_i(t)\bigr),
    \qquad i=1,\ldots,n,
\end{aligned}
\label{eq:general-perturbation-coordinates}
\end{equation}
and write
\[
    u(t)
    :=
    \bigl(u_1(t),\ldots,u_n(t)\bigr)^\top.
\]
The generalized perturbation and demodulation signals are
\begin{equation}
\begin{aligned}
    S_i(t)
    &:={}
    a u_i(t)
    =a\phi\bigl(\theta_i(t)\bigr),\\
    M_i(t)
    &:={}
    \frac{1}{a}m_i\bigl(u(t)\bigr),
    \qquad i=1,\ldots,n,
\end{aligned}
\label{eq:general-perturbation-demodulation-signals}
\end{equation}
where each \(m_i:[-1,1]^n\to\mathbb{R}\) is Lipschitz. With these
signals, the modified extremum seeking dynamics retain the form
\begin{equation}
    \dot{\hat x}(t)
    =
    -kJ\bigl(\hat x(t)+S(t)\bigr)M(t).
\label{eq:general-perturbation-ES}
\end{equation}

To describe the spatial averaging induced by \(\phi\), suppose that there
exists an integrable function
\(\rho:[-1,1]^n\to\mathbb{R}_{\geq 0}\) such that
\begin{equation}
    \frac{1}{(2\pi)^n}
    \int_{[0,2\pi]^n}
    h\bigl(\phi(\theta_1),\ldots,\phi(\theta_n)\bigr)
    \,\mathrm{d}\theta =
    \int_{[-1,1]^n}
    h(u)\rho(u)\,\mathrm{d}u
\label{eq:general-perturbation-density}
\end{equation}
for every Lipschitz function
\(h:[-1,1]^n\to\mathbb{R}\). Taking \(h\equiv 1\) in
\eqref{eq:general-perturbation-density} gives
\[
    \int_{[-1,1]^n}\rho(u)\,\mathrm{d}u=1.
\]

Let \(\kappa:[-1,1]^n\to\mathbb{R}_{\geq 0}\) be a continuous kernel
satisfying
\begin{equation}
    \int_{[-1,1]^n}\kappa(u)\,\mathrm{d}u=1.
\label{eq:general-matched-kernel-mass}
\end{equation}
Suppose that \(\kappa\) is absolutely continuous in each coordinate, that
its partial derivatives are integrable, and that
\begin{equation}
    \kappa(u)=0
    \quad\text{whenever}\quad
    u_i\in\{-1,1\}
\label{eq:general-matched-kernel-boundary}
\end{equation}
for every \(i=1,\ldots,n\). Finally, suppose that the demodulation
functions and the perturbation density satisfy
\begin{equation}
    m_i(u)\rho(u)
    =
    -\partial_i\kappa(u),
    \qquad i=1,\ldots,n,
\label{eq:general-perturbation-kernel-matching}
\end{equation}
for almost every \(u\in[-1,1]^n\).

For the kernel \(\kappa\), define the smoothed objective
\begin{equation}
    J_{a,\kappa}(x)
    :={}
    \int_{[-1,1]^n}
    J(x+au)\kappa(u)\,\mathrm{d}u.
\label{eq:general-smoothed-objective}
\end{equation}

\begin{theorem}[General Perturbation and Demodulation]
\label{thm:general-perturbation-demodulation}
Suppose that Assumptions~\ref{ass:objective-lipschitz} and~
\ref{ass:rationally-independent-frequencies} hold. Let
\(\phi:\mathbb{R}\to[-1,1]\) be Lipschitz and \(2\pi\)-periodic, and
let the signals \(S\) and \(M\) be defined by
\eqref{eq:general-perturbation-coordinates} and
\eqref{eq:general-perturbation-demodulation-signals}, where each
\(m_i:[-1,1]^n\to\mathbb{R}\) is Lipschitz.

Suppose that there exists a nonnegative integrable function
\(\rho:[-1,1]^n\to\mathbb{R}_{\geq 0}\) satisfying
\eqref{eq:general-perturbation-density} for every Lipschitz function
\(h:[-1,1]^n\to\mathbb{R}\). Let
\(\kappa:[-1,1]^n\to\mathbb{R}_{\geq 0}\) be continuous and satisfy
the unit-mass condition \eqref{eq:general-matched-kernel-mass} and the
boundary condition \eqref{eq:general-matched-kernel-boundary}. Suppose
also that \(\kappa\) is absolutely continuous in each coordinate, that
its almost-everywhere-defined partial derivatives are integrable, and
that the matching condition
\eqref{eq:general-perturbation-kernel-matching} holds. Define
\(J_{a,\kappa}\) by \eqref{eq:general-smoothed-objective}. Then the
following statements hold.

\begin{enumerate}
    \item For every \(t_0\in\mathbb{R}\), the following limit exists:
    \begin{equation}
    \begin{aligned}
        &\lim_{T\to\infty}
        \frac{1}{T}
        \int_{t_0}^{t_0+T}
        J\bigl(x+S(t)\bigr)M(t)\,\mathrm{d}t\\
        &\qquad =
        \int_{[-1,1]^n}
        \nabla J(x+au)\kappa(u)\,\mathrm{d}u\\
        &\qquad =
        \nabla J_{a,\kappa}(x),
    \end{aligned}
    \label{eq:general-perturbation-long-time-average}
    \end{equation}
    where \(\nabla J\) is understood almost everywhere. The convergence is uniform with respect to \(x\) on compact subsets of
\(\mathbb{R}^n\) and \(t_0\in\mathbb{R}\).

    \item Suppose, in addition, that \(x_{a,\kappa}\in\mathbb{R}^n\)
    is a GUAS equilibrium of the gradient flow
    \begin{equation}
        \frac{\mathrm{d}z}{\mathrm{d}s}
        =
        -\nabla J_{a,\kappa}(z).
    \label{eq:general-perturbation-gradient-flow}
    \end{equation}
    Then the conclusion of Theorem~\ref{thm:main-ES-result} holds for
    \eqref{eq:general-perturbation-ES}, with \(x_a\) replaced by
    \(x_{a,\kappa}\). In particular, \(x_{a,\kappa}\) is PGUAS with respect to $\varepsilon = k / \omega$, and the bound
    \eqref{eq:main-SPUAS-bound} holds with \(x_a\) replaced by
    \(x_{a,\kappa}\).
\end{enumerate}
\end{theorem}
%\vspace{1mm}

The theorem follows by repeating the basic analysis steps outlined in
Section~\ref{sec:analysis}, and the proof is included in Appendix~\ref{app:proof-general-design}.

\subsection{Interpretation of the Sinusoidal Design}
\label{subsec:sinusoidal-occupation-density}

The design is particularly transparent if the Kronecker--Weyl average is
first transformed into the spatial coordinates
\[
    u_i=\sin\theta_i,
    \qquad i=1,\ldots,n
\]
and integration by parts is then performed in these coordinates. For every
continuous \(h:[-1,1]^n\to\mathbb{R}\), Theorem~\ref{thm:equidistribution-independent-phases} gives
\begin{align}
\lim_{T\to\infty}
\frac{1}{T}\int_0^T
h\bigl(\sin(\hat\omega\tau)\bigr)
\,\mathrm{d}\tau
& = 
\frac{1}{(2\pi)^n}
\int_{[0,2\pi]^n}
h(\sin\theta)\,\mathrm{d}\theta
 \nonumber  \\&= 
\int_{[-1,1]^n}
h(u)\rho_A(u)\,\mathrm{d}u.
\label{eq:sinusoidal-long-time-average}
\end{align}
where
\begin{equation}
\rho_A(u)
=
\frac{1}{\pi^n}
\prod_{j=1}^n
\frac{1}{\sqrt{1-u_j^2}},
\qquad u\in(-1,1)^n.
\label{eq:arcsine-occupation-density}
\end{equation}
This density results from splitting each phase interval into the two
monotone branches of the sine function and then changing coordinates from
\(\theta\) to \(u\). Thus, the sinusoidal perturbation induces the
product-arcsine density \(\rho_A\) on \([-1,1]^n\). This particular density is
determined entirely by the perturbation and is present before a demodulation
signal is selected.

For the proposed demodulation, the corresponding spatial function is
\begin{equation}
    m_i(u)
    =
    2^n u_i
    \prod_{j\neq i}(1-u_j^2).
    \label{eq:spatial-selected-demodulator}
\end{equation}
Combining this expression with
\eqref{eq:arcsine-occupation-density} gives
\begin{align}
    m_i(u)\rho_A(u)
    &=
    \left(\frac{2}{\pi}\right)^n
    \frac{u_i}{\sqrt{1-u_i^2}}
    \prod_{j\neq i}\sqrt{1-u_j^2}
    \nonumber\\
    &=
    -\partial_i\kappa(u),
    \label{eq:sinusoidal-matching-check}
\end{align}
where \(\kappa\) is the product-semicircle kernel in
\eqref{eq:product-semicircle-kernel}. Hence, the matching condition
\eqref{eq:general-perturbation-kernel-matching} holds for every component.

The cross-coordinate cosine-squared factors in the demodulator are precisely
what converts the product-arcsine occupation density into the derivative of
the same product-semicircle kernel for every component. This common kernel
makes the averaged field the gradient of a single smoothed objective.

With the same sinusoidal perturbation, one may also choose another
admissible kernel \(\widetilde\kappa\) whenever
\[
    \widetilde m_i(u)
    =
    -\frac{\partial_i\widetilde\kappa(u)}{\rho_A(u)}
\]
can be defined at the boundary so that each \(\widetilde m_i\) is Lipschitz
on \([-1,1]^n\). The product-semicircle
kernel is particularly convenient because it produces the elementary
time-domain demodulator in \eqref{eq:design-demodulator}.

\begin{figure*}[!t]
    \centering
    \includegraphics[width=0.9\textwidth]
    {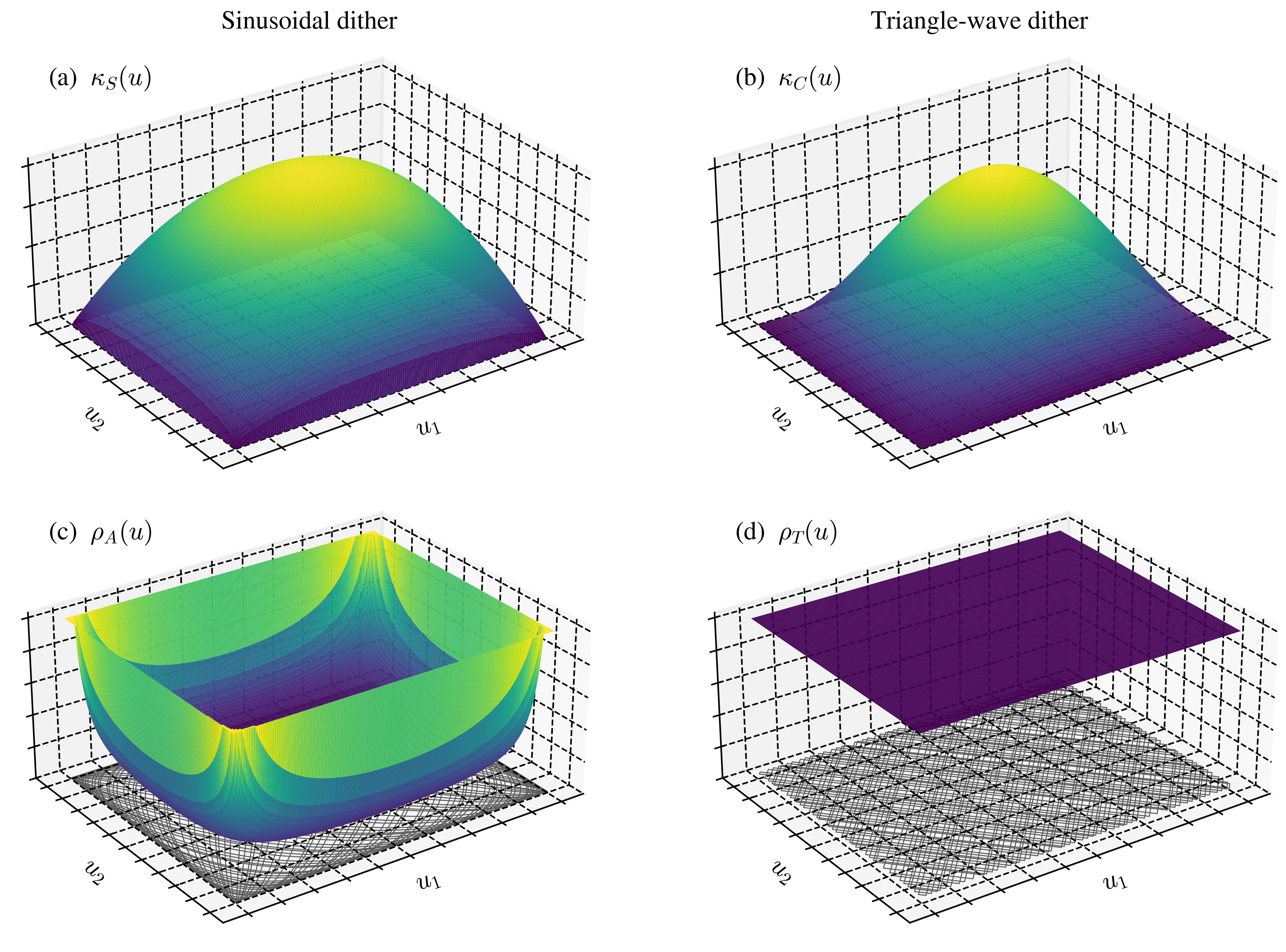}
\caption{Smoothing kernels (top) and occupation densities (bottom) for
two-dimensional sinusoidal (left) and triangle-wave (right) perturbations.
Black curves (bottom) show perturbation trajectories with rationally
independent frequencies \(\hat\omega=(1,\sqrt{2})\) which induce $\rho_A$ and $\rho_T$, respectively.}
    \label{fig:perturbation-density-kernel-comparison}
\end{figure*}

\subsection{Triangle Wave Example}
\label{subsec:other-perturbation-signals}

The same principle can guide the use of a different perturbation. First, use its
long-time average to determine its spatial density \(\rho(u)\) on the sampled
region. Next, choose a nonnegative, unit-mass smoothing kernel \(\kappa(u)\)
that vanishes on the boundary. Wherever \(\rho(u)>0\), select the spatial
demodulator according to
\begin{equation}
    m_i(u)\rho(u)
    =
    -\partial_i\kappa(u).
    \label{eq:general-occupation-matching}
\end{equation}
The resulting expression must then be converted back to a realizable
time-domain signal. This construction requires the perturbation to possess a
suitable spatial density and the ratio in
\eqref{eq:general-occupation-matching} to remain bounded and well defined.

A useful example is the triangle wave. We define
\begin{equation}
\phi(\theta)
:=
\frac{2}{\pi}\arcsin(\sin\theta),
\qquad
u_i=\phi(\theta_i).
\label{eq:normalized-triangle-wave}
\end{equation}
Following a procedure similar to that used to derive
\eqref{eq:multidimensional-semicircle-change}, shift each phase interval
to \([-\pi/2,3\pi/2]\) and split it into two monotone branches. Since
\(\lvert\mathrm{d}u_i\rvert=(2/\pi)\lvert\mathrm{d}\theta_i\rvert\), the
phase integral converts directly to the \(u\) coordinates. For rationally
independent component frequencies, the joint density is
\begin{equation}
\rho_T(u)=\frac{1}{2^n},
\qquad u\in(-1,1)^n.
\label{eq:triangle-wave-density}
\end{equation}
The matching condition
\eqref{eq:general-occupation-matching} consequently reduces to
\begin{equation}
m_i(u)
=
-2^n\partial_i\kappa(u).
\label{eq:triangle-wave-matching}
\end{equation}
Thus, kernels with bounded partial derivatives immediately produce
bounded triangle-wave demodulators.

For example, consider the product quadratic kernel
\begin{equation}
\kappa_Q(u)
=
\left(\frac{3}{4}\right)^n
\prod_{j=1}^n(1-u_j^2).
\label{eq:triangle-epanechnikov-kernel}
\end{equation}
This kernel is nonnegative, has unit mass on \([-1,1]^n\), and vanishes
on the boundary. From \eqref{eq:triangle-wave-matching}, its spatial
demodulator is
\begin{equation}
\begin{aligned}
m_{Q,i}(u)
=
2^{n+1}\left(\frac{3}{4}\right)^n u_i
\prod_{j\neq i}(1-u_j^2).
\end{aligned}
\label{eq:triangle-epanechnikov-demodulator}
\end{equation}
If \(\theta_j(t)=\omega\hat\omega_jt\), the corresponding time-domain
demodulator is
\begin{equation}
\begin{aligned}
M_{Q,i}(t)
=
\frac{2^{n+1}}{a}
\left(\frac{3}{4}\right)^n
\phi\bigl(\theta_i(t)\bigr)
\prod_{j\neq i}
\left[
1-\phi\bigl(\theta_j(t)\bigr)^2
\right].
\end{aligned}
\label{eq:triangle-epanechnikov-time}
\end{equation}
This signal is bounded and continuous, and is piecewise polynomial in
the triangle-wave coordinates.

Another convenient choice is the product-cosine kernel
\begin{equation}
\kappa_C(u)
=
\left(\frac{\pi}{4}\right)^n
\prod_{j=1}^n
\cos\left(\frac{\pi u_j}{2}\right).
\label{eq:triangle-cosine-kernel}
\end{equation}
It is also nonnegative, has unit mass, and vanishes on the boundary.
Its spatial demodulator is
\begin{equation}
\begin{aligned}
m_{C,i}(u)
=
\frac{\pi^{n+1}}{2^{n+1}}
\sin\left(\frac{\pi u_i}{2}\right)
\prod_{j\neq i}
\cos\left(\frac{\pi u_j}{2}\right).
\end{aligned}
\label{eq:triangle-cosine-demodulator}
\end{equation}
The identities
$
\sin\left(\pi\phi(\theta)/2\right)
=\sin\theta
$
and
$
\cos\left(\pi\phi(\theta)/2\right)
=|\cos\theta|
$
then give the particularly simple time-domain realization
\begin{equation}
\begin{aligned}
M_{C,i}(t)
=
\frac{\pi^{n+1}}{a\,2^{n+1}}
\sin\bigl(\theta_i(t)\bigr) 
\prod_{j\neq i}
\left|\cos\bigl(\theta_j(t)\bigr)\right|.
\end{aligned}
\label{eq:triangle-cosine-time}
\end{equation}
This demodulator is bounded and continuous. The product
quadratic kernel is especially simple in the spatial coordinates,
whereas the product-cosine kernel gives the cleaner phase-domain
realization.

By contrast, using the product-semicircle kernel with the triangle
wave would produce a factor of the form
\[
\frac{\phi(\theta_i)}
{\sqrt{1-\phi(\theta_i)^2}},
\]
which becomes unbounded when the triangle wave reaches \(\pm1\).
This illustrates why the perturbation and smoothing kernel should be chosen
together: although the matching rule is general, some combinations
lead to substantially more regular time-domain demodulators than
others.

Fig.~\ref{fig:perturbation-density-kernel-comparison} compares the occupation
densities and matched smoothing kernels of the two designs for \(n=2\).
The uniform occupation density makes the triangle-wave design particularly
flexible: any kernel satisfying the required boundary and normalization
conditions and having Lipschitz partial derivatives yields Lipschitz
demodulation signals directly through \(m_i=-2^n\partial_i\kappa\).
Unlike the sinusoidal design, the matching rule involves no division by a
coordinate-dependent occupation density.

\subsection{Why the Classical Demodulator Does Not Produce a Common Kernel}
\label{subsec:classical-demodulator-limitation}

For comparison, the classical multivariable demodulator is
\begin{equation}
    M_i^{\mathrm{cl}}(t)
    =
    \frac{2}{a}
    \sin(\omega\hat\omega_i t).
    \label{eq:classical-multivariable-demodulator}
\end{equation}
Repeating the integration by parts and coordinate transformation steps gives
\begin{equation}
    \bar F_i^{\mathrm{cl}}(x)
    =
    -\int_{[-1,1]^n}
    \partial_iJ(x+au)
    \kappa_i^{\mathrm{cl}}(u)\,\mathrm{d}u,
    \label{eq:classical-component-average}
\end{equation}
where
\begin{equation}
    \kappa_i^{\mathrm{cl}}(u)
    =
    \frac{2}{\pi}\sqrt{1-u_i^2}
    \prod_{j\neq i}
    \frac{1}{\pi\sqrt{1-u_j^2}}.
    \label{eq:classical-component-kernel}
\end{equation}
Each \(\kappa_i^{\mathrm{cl}}\) has unit mass, but the kernel depends on the
component \(i\). Consequently, for \(n>1\), the classical average dynamics do
not generally have the common-kernel form
\begin{equation}
    -\int_{[-1,1]^n}
    \nabla J(x+au)\kappa(u)\,\mathrm{d}u
    =
    -\nabla J_a(x),
    \label{eq:common-kernel-form-comparison}
\end{equation}
with a single scalar \(\kappa\). The classical average field therefore cannot generally be expressed as the
gradient of a single smoothed scalar objective for arbitrary \(J\). This does not mean that
the classical multivariable ES law cannot be stable; it means only that its
stability cannot generally be inferred from GUAS of the gradient flow of a
smoothed objective for $n>1$.

\section{Examples}
\label{sec:examples}
\subsection{Application in Nonlinear Programming}
\label{subsec:nonsmooth-penalty-example}

\begin{figure}[!t]
    \centering
    \includegraphics[width=\linewidth]
    {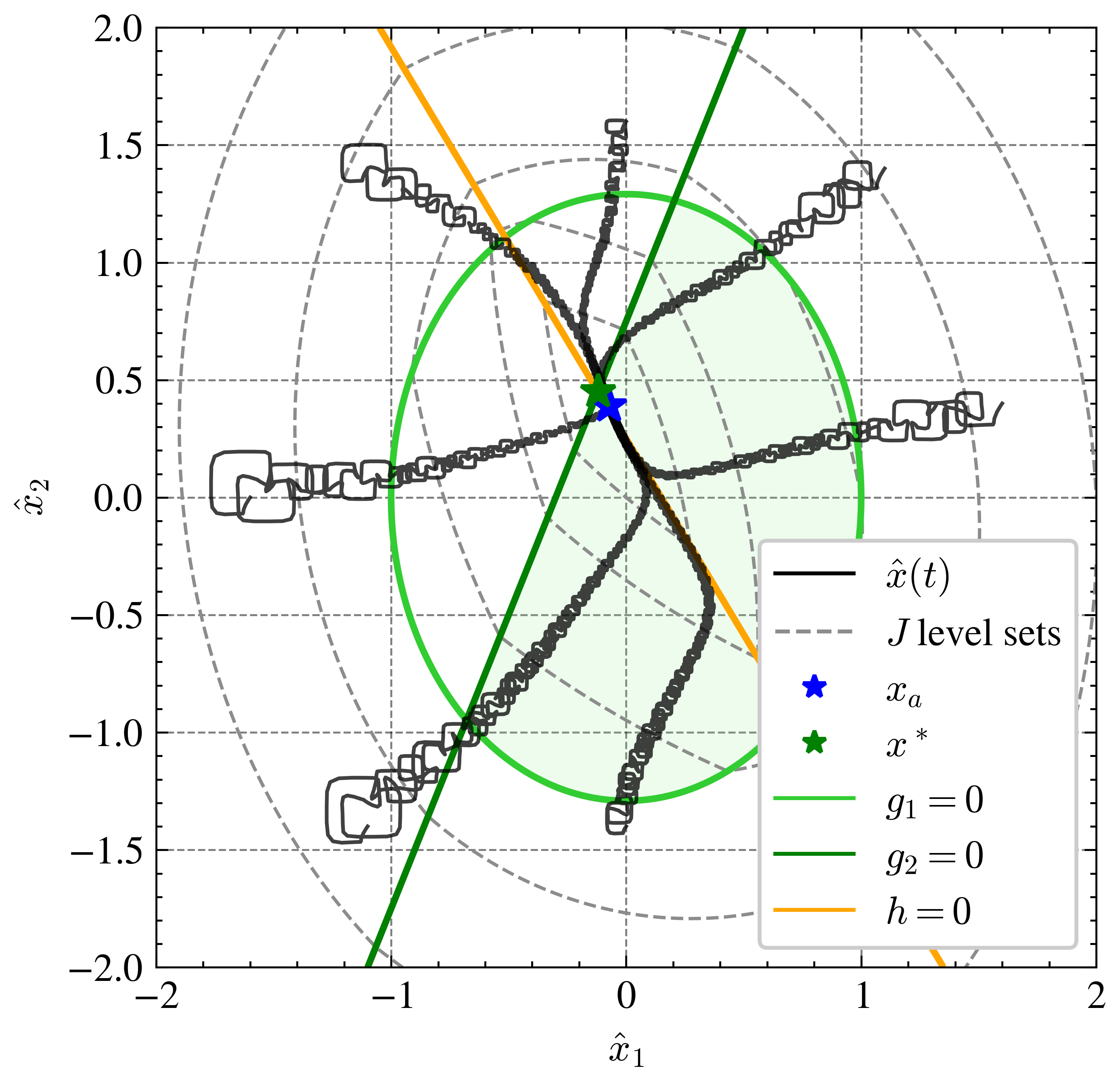}
    \caption{Parameter estimate trajectories for the nonsmooth penalty
    objective. The dashed gray curves are level sets of $J$, the light-green
    region satisfies $g_1(x)\leq0$ and $g_2(x)\leq0$, and the orange line
    denotes $h(x)=0$. The green and blue stars denote $x^*$ and $x_a$,
    respectively.}
    \label{fig:penalty-example-state-space}
\end{figure}

\begin{figure}[!t]
    \centering
    \includegraphics[width=\linewidth]
    {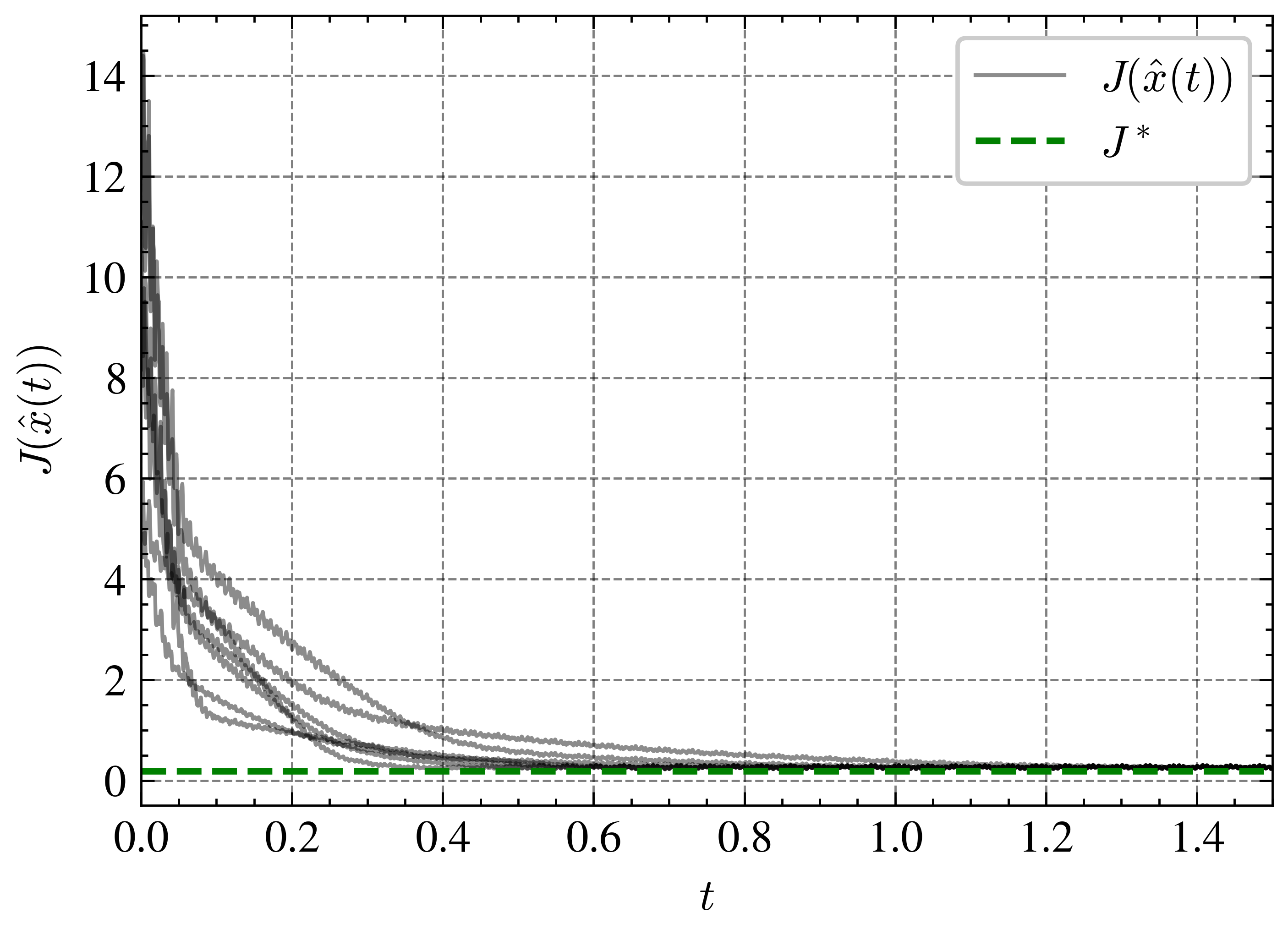}
    \caption{Objective values along the extremum seeking trajectories. The
    dashed green line denotes $J^*=J(x^*)$.}
    \label{fig:penalty-example-objective}
\end{figure}

We wish to minimize the locally Lipschitz objective
\[
    J(x)
    =f_0(x)+2.40|h(x)|+3.00[g_1(x)]_++2.00[g_2(x)]_+,
\]
where $[r]_+:=\max\{r,0\}$ and 
\begin{align*}
f_0(x)
&:=0.35(x_1+0.80)^2+0.55(x_2-0.70)^2,\\
h(x)
&:=x_1+0.60x_2-0.15,\\
g_1(x)
&:=x_1^2+0.60x_2^2-1,\\
g_2(x)
&:=-x_1+0.40x_2-0.30.
\end{align*}
The objective \(J\) has the standard nonsmooth exact-penalty form for minimizing \(f_0(x)\) over \(x\in\mathbb{R}^2\) subject to \(h(x)=0\), \(g_1(x)\leq0\), and \(g_2(x)\leq0\). The absolute-value term penalizes violation of the equality constraint, while the positive-part terms penalize violations of the inequalities. Under standard constraint qualifications, sufficiently large penalty weights recover solutions of the constrained problem \cite[Sec.~17.2]{nocedal2006numerical}.

Suppose that the objective and constraint-violation signals are only available
through measurements. For example, $f_0$ may be a measured performance signal, while $h$, $g_1$, and $g_2$ may represent measured operating limits or safety-related signals.
Their measured values can be combined online to form the scalar objective
$J$, which can then be supplied directly to the extremum seeking law. The
penalty terms encourage constraint satisfaction.

We apply \eqref{eq:design-perturbation}--\eqref{eq:proposed-ES-vector} with
\(n=2\), \(\hat\omega=(1,\sqrt{2})\), \(k=1\), \(a=0.25\), and
\(\omega=1000\). Forward Euler integration is performed over
\(t\in[0,1.5]\) with \(\Delta t=1.77\times10^{-4}\).
The perturbation components are
\begin{equation}
    S_1(t)=a\sin(\omega t),
    \qquad
    S_2(t)=a\sin(\sqrt{2}\,\omega t),
\label{eq:penalty-example-perturbation}
\end{equation}
and the demodulation components are
\begin{equation}
\begin{aligned}
    M_1(t)
    &=\frac{4}{a}\sin(\omega t)
      \cos^2(\sqrt{2}\,\omega t),\\
    M_2(t)
    &=\frac{4}{a}\sin(\sqrt{2}\,\omega t)
      \cos^2(\omega t).
\end{aligned}
\label{eq:penalty-example-demodulation}
\end{equation}

For this particular objective, Assumption~\ref{ass:average-system-GUAS} is
straightforward to verify. The Hessian of $f_0$ is
$\operatorname{diag}(0.7,1.1)$, so $f_0$ is $0.7$-strongly convex. Moreover,
$|h|$ is convex because $h$ is affine, while $[g_1]_+$ and $[g_2]_+$ are
convex because they are pointwise maxima of convex functions
\cite[Sec.~3.2]{boyd2004convex}. Consequently, $J$ is $0.7$-strongly convex.
Each translate $x\mapsto J(x+au)$ has the same strong-convexity parameter,
and a nonnegative unit-mass integral preserves this parameter
\cite[Secs.~3.2.1 and~9.1.2]{boyd2004convex}. Hence, $J_a$ is also
$0.7$-strongly convex. Its gradient flow is therefore GUAS at its unique
minimizer, and the hypothesis of Theorem~\ref{thm:main-ES-result} is
satisfied. For the selected penalty weights, the minimizer of $J$ is
$x^*=(-0.12,0.45)^\top$, while the minimizer of the smoothed objective is
approximately $x_a=(-0.0726,0.3891)^\top$.

Fig.~\ref{fig:penalty-example-state-space} shows the parameter trajectories
from several initial conditions. The trajectories approach a neighborhood of $x_a$, while the displacement
between $x_a$ and $x^*$ illustrates the bias introduced by the nonzero
smoothing amplitude. Fig.~\ref{fig:penalty-example-objective} shows the
corresponding values of $J(\hat x(t))$. These values approach a neighborhood
of $J(x_a)$, which lies close to $J^*:=J(x^*)$ for the selected amplitude. In this example, the spatial averaging inherent in the design biases
\(x_a\) toward the interior of
\(\{x:g_1(x)\leq0,\ g_2(x)\leq0\}\) because, near a constraint boundary
\(g_i=0\), perturbation samples with \(g_i>0\) activate the penalty and
contribute the locally inward descent direction \(-\nabla g_i\).
This effect may be useful when \(g_1\) and \(g_2\) represent
safety-related measurement signals.

\subsection{Removal of Local Minima by Smoothing}
\label{subsec:rastrigin-example}

Consider the two-dimensional Rastrigin objective
\begin{equation}
    J(x)
    =
    20+x_1^2+x_2^2
    -10\cos(2\pi x_1)
    -10\cos(2\pi x_2).
    \label{eq:rastrigin-objective}
\end{equation}
The function is smooth and nonconvex, with many local minima arranged
throughout the state space. Its unique global minimizer is $x^*=0$, at which
$J(x^*)=0$.
\begin{figure}[!t]
    \centering
    \includegraphics[width=0.95\linewidth]
    {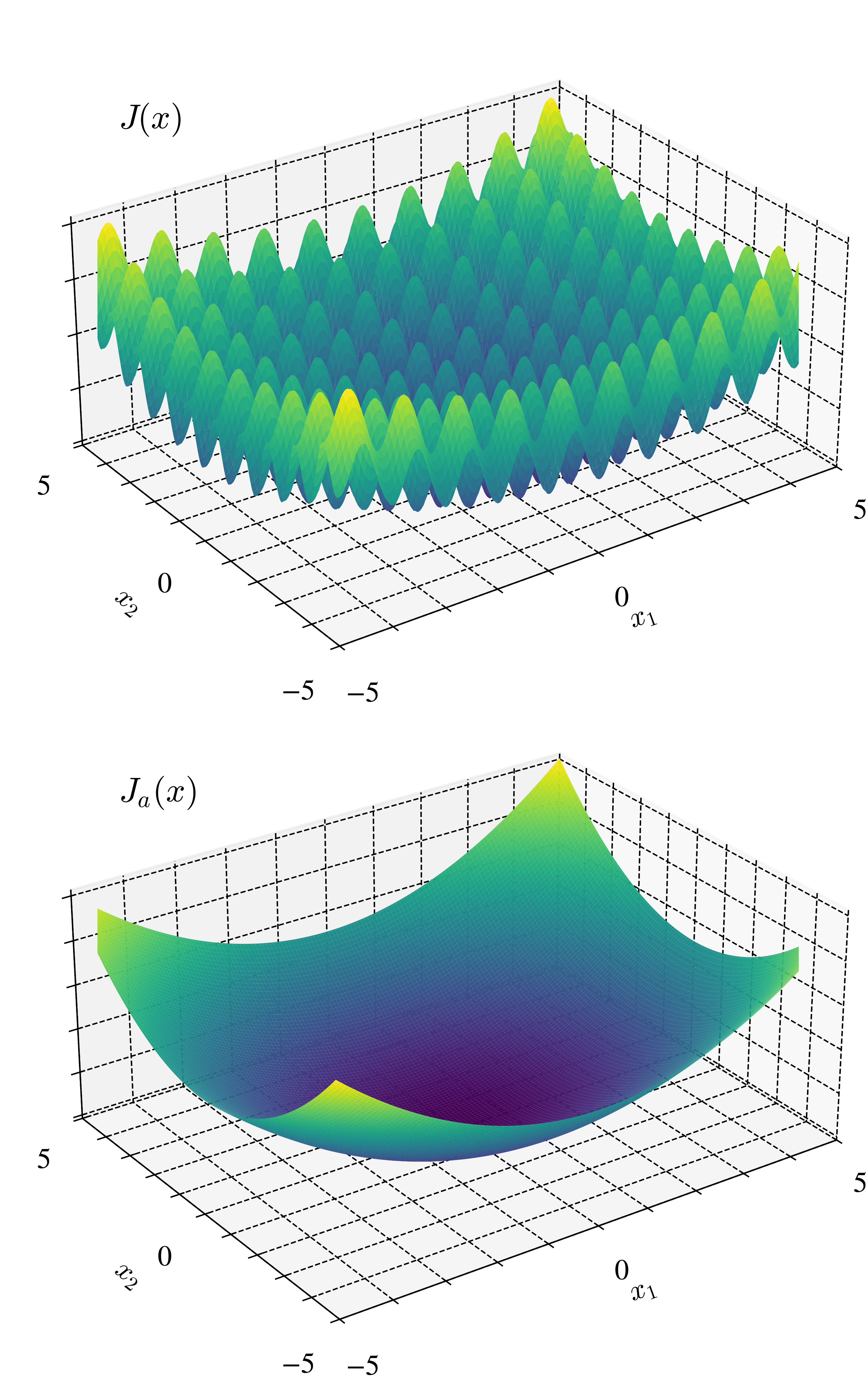}
    \caption{The Rastrigin objective \(J\) (top) and the smoothed objective
    \(J_a\) generated by the product-cosine kernel with \(a=0.75\) (bottom).
    At this amplitude, the oscillatory terms average to zero and \(J_a\) is
    exactly quadratic.}
    \label{fig:rastrigin-surfaces}
\end{figure}

\begin{figure}[!t]
    \centering
    \includegraphics[width=\linewidth]
    {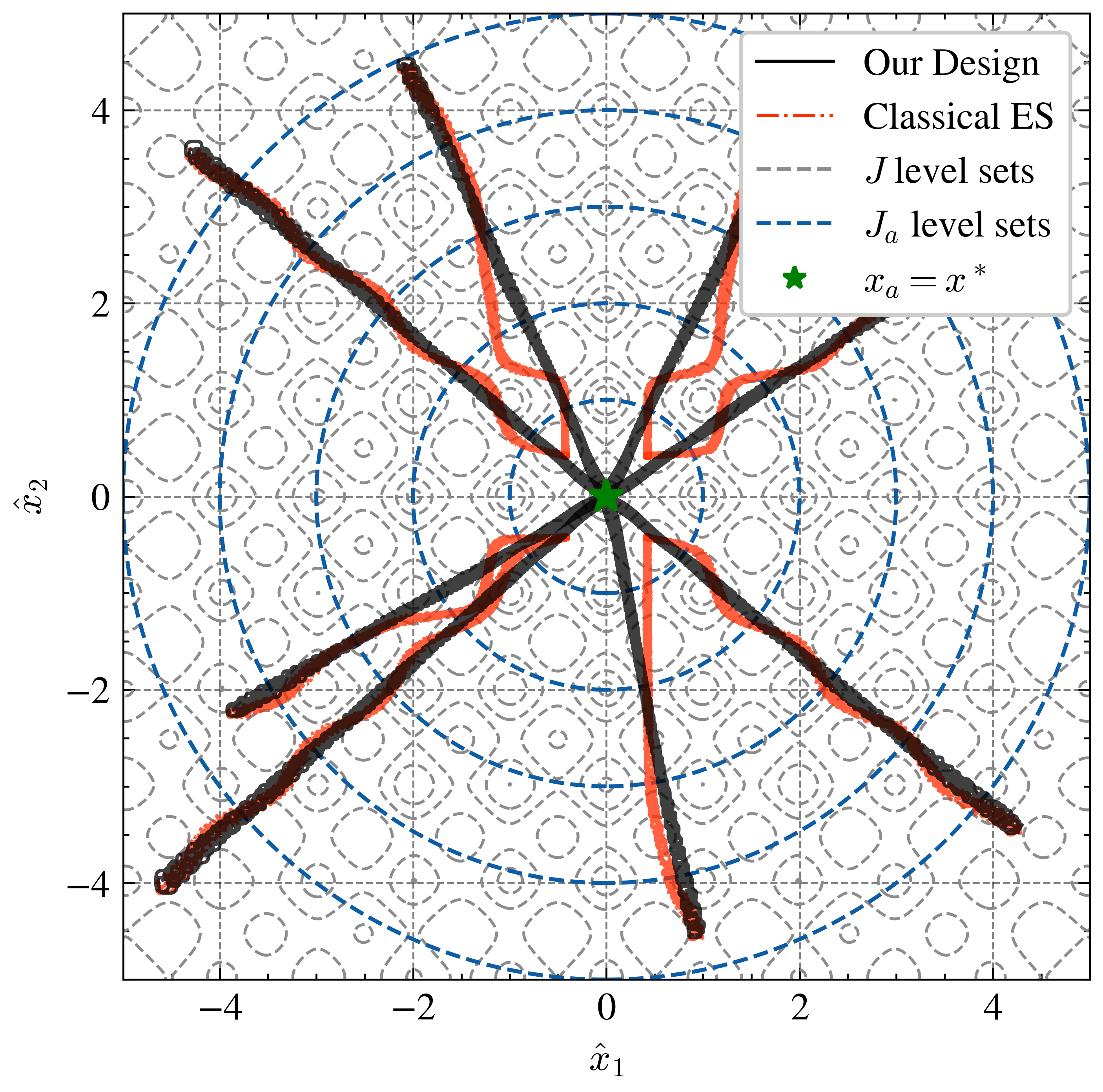}
    \caption{Triangle-wave extremum seeking trajectories for the Rastrigin
    objective. The dashed gray and blue curves are level sets of $J$ and
    $J_a$ respectively, and the green star denotes their common minimizer
    $x_a=x^*=0$. We plot the trajectories of our design in black, and the classical design in orange.}
    \label{fig:rastrigin-triangle-trajectories}
\end{figure}

We apply the triangle-wave design from
Section~\ref{subsec:other-perturbation-signals}. For the relative frequencies
$\hat\omega_1=1$ and $\hat\omega_2=\sqrt{2}$, the perturbation components are
\begin{equation}
\begin{aligned}
    S_1(t)
    &=
    \frac{2a}{\pi}\arcsin\bigl(\sin(\omega t)\bigr),\\
    S_2(t)
    &=
    \frac{2a}{\pi}
    \arcsin\bigl(\sin(\sqrt{2}\,\omega t)\bigr),
\end{aligned}
\label{eq:rastrigin-triangle-perturbation}
\end{equation}
and the product-cosine demodulation components are
\begin{equation}
\begin{aligned}
    M_1(t)
    &=
    \frac{\pi^3}{8a}
    \sin(\omega t)
    \left|\cos(\sqrt{2}\,\omega t)\right|,\\
    M_2(t)
    &=
    \frac{\pi^3}{8a}
    \sin(\sqrt{2}\,\omega t)
    \left|\cos(\omega t)\right|.
\end{aligned}
\label{eq:rastrigin-triangle-demodulation}
\end{equation}
The simulations use $k=1$, $a=0.75$, and $\omega=2000$. The trajectories are simulated over
$t\in[0,3]$ using forward Euler integration with $\Delta t= 8.85 \times 10^{-5}$.

As shown in Section~\ref{subsec:other-perturbation-signals}, this perturbation and
demodulation pair produces the product-cosine kernel
\begin{equation}
    \kappa_C(u)
    =
    \left(\frac{\pi}{4}\right)^2
    \cos\left(\frac{\pi u_1}{2}\right)
    \cos\left(\frac{\pi u_2}{2}\right),
    \qquad u\in[-1,1]^2.
    \label{eq:rastrigin-cosine-kernel}
\end{equation}
Accordingly, the smoothed objective is
\begin{equation}
    J_a(x)
    =
    \int_{[-1,1]^2}
    J(x+au)\kappa_C(u)\,\mathrm{d}u.
    \label{eq:rastrigin-smoothed-objective}
\end{equation}
To evaluate this integral, define
\begin{equation}
\begin{aligned}
    \sigma(a)
    &:=
    \frac{\pi}{4}
    \int_{-1}^{1}
    \cos(2\pi au)
    \cos\left(\frac{\pi u}{2}\right)\,\mathrm{d}u\\
    &=
    \begin{cases}
        \displaystyle
        -\frac{\cos(2\pi a)}{16a^2-1},
        & a\neq\frac14,\\[2ex]
        \displaystyle
        \frac{\pi}{4},
        & a=\frac14.
    \end{cases}
\end{aligned}
\label{eq:rastrigin-sigma}
\end{equation}
Also note the identity
\[
    \frac{\pi}{4}
    \int_{-1}^{1}
    u^2\cos\left(\frac{\pi u}{2}\right)\,\mathrm{d}u
    =
    1-\frac{8}{\pi^2}.
\]
Substitution of the above identities into \eqref{eq:rastrigin-smoothed-objective} gives
\begin{multline}
    J_a(x)
    =
    20+2a^2\left(1-\frac{8}{\pi^2}\right)
    +x_1^2+x_2^2 \\
    {}-10\sigma(a)
    \left[
        \cos(2\pi x_1)+\cos(2\pi x_2)
    \right].
    \label{eq:rastrigin-smoothed-objective-reduced}
\end{multline}
At the selected amplitude $a=0.75$,
$
    \sigma(0.75)
    % =
    % -\frac{\cos(3\pi/2)}{8}
    =0.
$
Consequently, the smoothing operation removes the entire oscillatory part of
$J$ and yields the quadratic function
\begin{equation}
    J_a(x)
    =
    20+\frac{9}{8}\left(1-\frac{8}{\pi^2}\right)
    +x_1^2+x_2^2.
    \label{eq:rastrigin-smoothed-objective-quadratic}
\end{equation}
Thus, $J_a$ is $2$-strongly convex, $x_a=x^*=0$, and the average system in
slow time reduces to
\begin{equation}
    \frac{\mathrm{d}z}{\mathrm{d}s}
    =-2z.
    \label{eq:rastrigin-average-system}
\end{equation}
In this example, smoothing does not merely reduce the influence of the local
minima: it removes every nonzero local minimum exactly while preserving the
global minimizer of the original objective. Although $a=0.75$ is selected to
cancel the oscillatory terms exactly, the same qualitative behavior holds for
all sufficiently large amplitudes. For $a>1/4$,
$
|\sigma(a)|
\leq
\frac{1}{16a^2-1},
$
so the corresponding averaged dynamics have similar
global convergence behavior for all sufficiently large $a$. 

Fig.~\ref{fig:rastrigin-surfaces} compares the original and smoothed
objectives over $[-5,5]^2$. The repeated wells of $J$ are visible in the
upper panel, whereas the lower panel shows that the smoothed objective is
a single quadratic bowl. Fig.~\ref{fig:rastrigin-triangle-trajectories} shows the trajectories of the
triangle-wave extremum seeking system from eight initial conditions. Apart from their small fast oscillations, the trajectories take
strikingly direct paths toward the origin despite passing through many basins
of the original objective. This behavior follows from
\eqref{eq:rastrigin-average-system}: its average solutions are
$z(s)=\exp(-2s)z(0)$ and therefore remain on straight rays leading to the
origin. The simulation illustrates how the trajectories
inherit the global behavior of the gradient flow of the smoothed
objective rather than becoming trapped near the local minima of \(J\).

For comparison with a conventional periodic triangle-wave ES design, we
also implement its classical perturbation--demodulation analogue. Its two state equations are
\[
\begin{aligned}
    \dot{\hat x}_1(t)
    &=
    -kJ\bigl(\hat x(t)+S(t)\bigr)M_1(t),\\
    \dot{\hat x}_2(t)
    &=
    -kJ\bigl(\hat x(t)+S(t)\bigr)M_2(t),
\end{aligned}
\]
where the perturbation and demodulation components are
\[
\begin{aligned}
    S_i(t)
    &=
    \frac{2a}{\pi}
    \arcsin\bigl(\sin(\hat\omega_i\omega t)\bigr),\\
    M_i(t)
    &=
    \frac{6}{a\pi}
    \arcsin\bigl(\sin(\hat\omega_i\omega t)\bigr),
    \qquad i\in\{1,2\}.
\end{aligned}
\]
We take $\hat\omega_1=1$ and $\hat\omega_2=17/12\approx\sqrt{2}$ and use
the same initial conditions and the same values of $a$, $k$, $\omega$, and $\Delta t$ as
our design.

The classical analysis technique follows from a Taylor expansion and a periodic
time average. Introduce the fast time $\tau=\omega t$ and define the unit
triangle waves
\[
    q_i(\tau)
    =
    \frac{2}{\pi}
    \arcsin\bigl(\sin(\hat\omega_i\tau)\bigr),
    \qquad i\in\{1,2\}.
\]
Thus, $S=aq$ and $M=(3/a)q$. For the selected relative frequencies, $q$ is
periodic with common period $T=24\pi$ and satisfies
\[
    \frac{1}{T}
    \int_0^{T}q(\tau)\,\mathrm{d}\tau
    =0,
    \qquad
    \frac{1}{T}
    \int_0^{T}
    q(\tau)q(\tau)^\top\,\mathrm{d}\tau
    =
    \frac{1}{3}I.
\]
Since $J$ is smooth,
\[
    J\bigl(z+aq(\tau)\bigr)
    =
    J(z)
    +a q(\tau)^\top\nabla J(z)
    +O(a^2).
\]
Substituting this expansion into the periodic average gives
\[
\begin{aligned}
    \frac{\mathrm{d}z}{\mathrm{d}\tau}
    &=
    -\frac{k}{\omega}\frac{3}{a}
    \frac{1}{T}
    \int_0^{T}
    J\bigl(z+aq(\tau)\bigr)q(\tau)\,\mathrm{d}\tau\\
    &=
    -\frac{k}{\omega}
    \left(\nabla J(z)+O(a)\right).
\end{aligned}
\]
This calculation gives the standard small-amplitude interpretation of the
classical design and confirms the normalization \(3/a\) of its demodulation
signal. By contrast, the matched design has an averaged vector field exactly
equal to \(-\nabla J_a\) at the fixed perturbation amplitude. Moreover, even if the
classical triangle-wave demodulator were paired with rationally independent
frequencies, its spatial form \(m_i(u)=3u_i\) would not satisfy
\eqref{eq:triangle-wave-matching} for the product-cosine kernel, because
\(-\partial_i\kappa_C(u)\) contains factors depending on the coordinates
\(u_j\), \(j\neq i\). As shown in
Fig.~\ref{fig:rastrigin-triangle-trajectories}, the classical trajectories
initially progress toward the origin but are eventually trapped near local
minima of \(J\), in contrast to the trajectories of the matched design.

\section{Conclusion}
\label{sec:conclusion}
This paper developed a multivariable extremum seeking method for locally
Lipschitz objectives. Rationally independent perturbation frequencies and matched
demodulation signals produce averaged dynamics exactly equal to the negative
gradient of a smoothed objective. When the
gradient flow of the smoothed objective is globally uniformly asymptotically stable, the
extremum seeking system is practically globally uniformly asymptotically stable. A general matching relation between the perturbation occupation density,
demodulation signals, and smoothing kernel explains the sinusoidal
construction and yields alternative designs, including a
triangle-wave design. The present results concern direct evaluations of a static objective.
Extensions to estimator filters and dynamic plants, with the associated time scale separation requirements, remain for future work.

%\section*{Appendix}
\appendix
\subsection{Differentiation Under the Integral}
\label{app:diff-under-integral}
%\subsection{Differentiation Under the Integral Sign}
The following result demonstrates that differentiation may be passed under the integral sign in a specific integral form.

\begin{proposition}
\label{prop:diff-under-integral}
Let $K\subset\mathbb{R}^n$ be compact, let $a>0$, and let
$\kappa:K\to\mathbb{R}_{\geq0}$ satisfy
\[
    \int_K \kappa(u)\,\mathrm{d}u=1.
\]
Suppose that $J:\mathbb{R}^n\to\mathbb{R}$ is locally Lipschitz, and define
\[
    J_a(x)
    :=
    \int_K J(x+au)\kappa(u)\,\mathrm{d}u.
\]
Then $J_a$ is differentiable on $\mathbb{R}^n$, with
\[
    \nabla J_a(x)
    =
    \int_K \nabla J(x+au)\kappa(u)\,\mathrm{d}u,
\]
where $\nabla J$ is understood almost everywhere.
\end{proposition}

\begin{proof}
Fix $x\in\mathbb{R}^n$. Since $J$ is locally Lipschitz, the map
$
    u\mapsto J(x+au)
$
is locally Lipschitz. By Rademacher's theorem
\cite[Theorem~3.2]{evans1991measure}, it is differentiable for almost
every $u\in K$. Since $a>0$, this is equivalent to differentiability of
$J$ at $x+au$.

For all sufficiently small $h\in\mathbb{R}^n$, the points $x+au$ and
$x+h+au$, with $u\in K$, lie in a fixed compact set. Let $L$ be a
Lipschitz constant for $J$ on this set. At every point where $J$ is
differentiable, one also has
$
    \|\nabla J(x+au)\|\leq L.
$
So, for almost every $u\in K$, the triangle inequality implies
\begin{multline*}
\frac{
\left|
J(x+h+au)-J(x+au)
-\nabla J(x+au)^\top h
\right|}
{\|h\|}
\kappa(u) \\
\leq 2L\kappa(u).
\end{multline*}
The left-hand side converges to zero as $h\to0$ for almost every
$u\in K$ by the definition of differentiability
\cite[Definition~3.2]{evans1991measure}, and $2L\kappa$ is integrable. The dominated convergence theorem
\cite[Theorem~2.24]{folland1999real} and the triangle
inequality therefore give
\begin{align*}
\lim_{h\to0}
\frac{1}{\|h\|}
\Bigg|
J_a(x+h) & -J_a(x)
- \\
&\left(
\int_K\nabla J(x+au)\kappa(u)\,\mathrm{d}u
\right)^\top h
\Bigg|
=0.
\end{align*}
Thus $J_a$ is differentiable at $x$ with the stated gradient. Since
$x$ was arbitrary, the result holds on $\mathbb{R}^n$.
\end{proof}

\subsection{GUAS and PGUAS Bounds}
\label{app:KL-bounds}

The following proposition is a consequence of
\cite[Theorem~1]{moreau2000practical} and expresses its conclusion in the
$\mathcal{KL}$ form of Definition~\ref{def:PGUAS}. In particular, the same
comparison function $\beta\in\mathcal{KL}$ that establishes GUAS of the
limiting system may be used in the practical estimate for the
$\varepsilon$-dependent system.

\begin{proposition}
\label{prop:moreau-KL-bound}
Suppose Hypotheses~1 and~2 of \cite{moreau2000practical} hold for
\eqref{eq:practical-parameter-system} and
\eqref{eq:practical-limit-system}, and suppose that the origin of
\eqref{eq:practical-limit-system} is GUAS with some
$\beta\in\mathcal{KL}$. Then the origin of
\eqref{eq:practical-parameter-system} is PGUAS in the sense of
Definition~\ref{def:PGUAS}, with the same comparison function $\beta$.
\end{proposition}

\begin{proof}
Fix $\Delta,\nu>0$. By
\cite[Theorem~1]{moreau2000practical}, specifically the practical global
uniform attractivity property in
\cite[Definition~2, Condition~3]{moreau2000practical}, applied with
\[
    c_1=\Delta+1,
    \qquad
    c_2=\nu,
\]
there exist $T>0$ and $\eta_a\in(0,\varepsilon_0]$ such that, whenever
$\|x_0\|\leq\Delta$ and $0<\varepsilon<\eta_a$, the corresponding
solution is forward complete and
\[
    \|x(t,\varepsilon)\|<\nu,
    \qquad
    t\geq t_0+T.
\]

Since the limiting system is GUAS, its solutions are forward complete.
Hypothesis~2 of \cite{moreau2000practical}, applied with
\[
    K=\{x_0\in\mathbb{R}^n:\|x_0\|\leq\Delta\},
    \qquad
    d=\nu,
\]
and the horizon $T$ selected above, gives
$\eta_c\in(0,\varepsilon_0]$ such that
\[
    \|x(t,\varepsilon)-z(t)\|<\nu,
    \qquad
    t_0\leq t\leq t_0+T,
\]
whenever $0<\varepsilon<\eta_c$. Hence,
\[
    \|x(t,\varepsilon)\|
    \leq
    \|z(t)\|+\|x(t,\varepsilon)-z(t)\|
    \leq
    \beta\bigl(\|x_0\|,t-t_0\bigr)+\nu,
\]
for $t_0\leq t\leq t_0+T.$
Define
\(
    \varepsilon^*(\Delta,\nu)
    :=
    \min\{\eta_a,\eta_c\}.
\)
For $t\geq t_0+T$ and $0<\varepsilon<\varepsilon^*(\Delta,\nu)$,
practical global uniform attractivity gives
\[
    \|x(t,\varepsilon)\|
    <
    \nu
    \leq
    \beta\bigl(\|x_0\|,t-t_0\bigr)+\nu.
\]
Combining the two time intervals gives the required estimate for every
$t\geq t_0$. The corresponding solutions are forward complete because
$\varepsilon^*(\Delta,\nu)\leq\eta_a$.
\end{proof}

\subsection{Proof of Corollary~\ref{cor:uniform-KW-average}}
\label{app:uniform-KW-average}

\begin{proof}
Fix a compact set \(K\subset\mathbb{R}^m\), and define
\[
    A_T(x,\psi)
    :=
    \frac{1}{T}\int_0^T
    H(x,\xi\tau+\psi)\,\mathrm{d}\tau.
\]
For each fixed
\((x,\psi)\in K\times[0,2\pi]^r\), apply Theorem~\ref{thm:equidistribution-independent-phases}
componentwise to the continuous periodic function
\[
    u\mapsto H(x,u+\psi).
\]
This gives
\[
\begin{aligned}
    \lim_{T\to\infty}A_T(x,\psi)
    &=
    \frac{1}{(2\pi)^r}
    \int_{[0,2\pi]^r}
    H(x,u+\psi)\,\mathrm{d}u\\
    &=
    \frac{1}{(2\pi)^r}
    \int_{[0,2\pi]^r}
    H(x,u)\,\mathrm{d}u
    =
    \bar H(x),
\end{aligned}
\]
where the second equality follows from periodicity and the change of integration variables $\bar u = u + \psi$.

Continuity, compactness of \(K\), and periodicity imply that \(H\) is
uniformly continuous on \(K\times\mathbb{R}^r\). We now show that
\(\{A_T:T>0\}\) is equicontinuous on
\(K\times[0,2\pi]^r\). Given \(\eta>0\), choose \(\delta>0\) such that
\[
    \|(x_1,v_1)-(x_2,v_2)\|<\delta
    \quad\Longrightarrow\quad
    \|H(x_1,v_1)-H(x_2,v_2)\|<\eta.
\]
If
\[
    \|(x_1,\psi_1)-(x_2,\psi_2)\|<\delta,
\]
then, for every \(\tau\geq0\),
\[
    \left\|
        (x_1,\xi\tau+\psi_1)
        -
        (x_2,\xi\tau+\psi_2)
    \right\|
    =
    \|(x_1,\psi_1)-(x_2,\psi_2)\|
    <\delta.
\]
Therefore, for every \(T>0\),
\[
\begin{aligned}
    & \|A_T(x_1,\psi_1)-A_T(x_2,\psi_2)\| \\
    &\qquad \leq
    \frac{1}{T}\int_0^T
    \left\|
        H(x_1,\xi\tau+\psi_1)
        -
        H(x_2,\xi\tau+\psi_2)
    \right\|\mathrm{d}\tau\\
        &\qquad <
    \frac{1}{T}\int_0^T
     \eta \, \mathrm{d}\tau = \eta .
\end{aligned}
\]
Thus the family \(\{A_T:T>0\}\) is equicontinuous, since \(\delta\) is
independent of \(T\).

Let \(T_j\to\infty\) be any sequence. The sequence
\(\{A_{T_j}\}_{j=1}^{\infty}\) is equicontinuous on the compact set
\(K\times[0,2\pi]^r\) and converges pointwise to
\(\bar H(x)\). Therefore, by
\cite[Chapter~7, Exercise~16]{rudin1976principles},
\(A_{T_j}\) converges uniformly to \(\bar H\).
Since this holds for every sequence \(T_j\to\infty\),
\(A_T\) converges uniformly to \(\bar H\) as \(T\to\infty\).

Finally, for every \(\tau_0\in\mathbb{R}\) and
\(\varphi\in[0,2\pi]^r\), changing variables
\(s=\tau_0+r\) gives
\[
\begin{aligned}
    \frac{1}{T}
    \int_{\tau_0}^{\tau_0+T}
    H(x,\xi s+\varphi)\,\mathrm{d}s
    &=
    \frac{1}{T}
    \int_0^T
    H\bigl(x,\xi r+\xi\tau_0+\varphi\bigr)
    \,\mathrm{d}r\\
    &=
    A_T(x,\xi\tau_0+\varphi).
\end{aligned}
\]
By periodicity, \(\xi\tau_0+\varphi\) may be reduced componentwise
modulo \(2\pi\). The uniform convergence of \(A_T\) with respect to its
second argument therefore proves
\eqref{eq:uniform-KW-average} uniformly in \(x\), \(\varphi\), and
\(\tau_0\).
\end{proof}

\subsection{Proof of Theorem~\ref{thm:general-perturbation-demodulation}}
\label{app:proof-general-design}
\begin{proof}
Let
\(\mathcal U:=[-1,1]^n\), \(\Theta:=[0,2\pi]^n\), and
\(\Phi(\theta):=(\phi(\theta_1),\ldots,\phi(\theta_n))^\top\). For
\(i=1,\ldots,n\), define
\[
    G_i(x,\theta)
    :=
    \frac{1}{a}
    J\bigl(x+a\Phi(\theta)\bigr)
    m_i\bigl(\Phi(\theta)\bigr).
\]
Then
\(J(x+S(t))M_i(t)=G_i(x,\omega\hat\omega t)\).
The function \(G_i\) is continuous, \(2\pi\)-periodic in every phase
variable, and locally Lipschitz in \(x\), uniformly in \(\theta\) on
compact state sets. Since \(\omega\hat\omega\) is rationally
independent, Corollary~\ref{cor:uniform-KW-average} gives the long-time
average
\[
    \bar G_i(x)
    =
    \frac{1}{(2\pi)^n}
    \int_{\Theta}G_i(x,\theta)\,\mathrm{d}\theta,
\]
with convergence uniform in \(x\) on compact sets and in
\(t_0\in\mathbb{R}\).

For almost every fixed \(u_{-i}\), the one-dimensional functions
\(u_i\mapsto J(x+au)\) and \(u_i\mapsto\kappa(u)\) are absolutely
continuous on \([-1,1]\), with
\(a\partial_iJ(x+au)\) almost everywhere, and $\kappa(u) = 0$ at
\(u_i=-1\) and \(u_i=1\) by
\eqref{eq:general-matched-kernel-boundary}. Therefore,
\eqref{eq:general-perturbation-density},
\eqref{eq:general-perturbation-kernel-matching}, integration by parts in
\(u_i\), and Fubini's theorem give
\[
\begin{aligned}
    \bar G_i(x)
    &=
    \frac{1}{a}\int_{\mathcal U}
    J(x+au)m_i(u)\rho(u)\,\mathrm{d}u\\
    &=
    -\frac{1}{a}\int_{\mathcal U}
    J(x+au)\partial_i\kappa(u)\,\mathrm{d}u\\
    &=
    \int_{\mathcal U}
    \partial_iJ(x+au)\kappa(u)\,\mathrm{d}u.
\end{aligned}
\]
Combining the components and applying
Proposition~\ref{prop:diff-under-integral} proves
\[
    \bar G(x)
    =
    \nabla J_{a,\kappa}(x),
\]
and hence \eqref{eq:general-perturbation-long-time-average}. The first line
of the preceding calculation also shows that
\(\nabla J_{a,\kappa}\) is locally Lipschitz, since \(J\) is locally
Lipschitz and each \(m_i\) is bounded on \(\mathcal U\).

For the stability statement, set \(\tau=\omega t\) and
\(\varepsilon=k/\omega\). With \(Q(x,\theta):=-G(x,\theta)\), the
extremum seeking system becomes
\(\mathrm{d}\hat x/\mathrm{d}\tau
=\varepsilon Q(\hat x,\hat\omega\tau)\), whose averaged vector field is
\(-\nabla J_{a,\kappa}\). The uniform convergence established above
and the stated local Lipschitz properties verify the hypotheses of
Theorem~\ref{thm:Sanders-general-averaging}. GUAS of
\eqref{eq:general-perturbation-gradient-flow} ensures its solutions are bounded on finite time intervals. Thus, with \(s=\varepsilon\tau=kt\) and
\(s_0=\varepsilon\tau_0=kt_0\), the extremum-seeking trajectories
converge uniformly to those of
\eqref{eq:general-perturbation-gradient-flow} on every fixed interval
\(s_0\leq s\leq s_0+L\), uniformly over compact sets of initial
conditions and the initial time.

The two conditions of
Theorem~\ref{thm:practical-stability} are therefore satisfied.
Applying that theorem after translating \(x_{a,\kappa}\) to the origin
proves that \(x_{a,\kappa}\) is PGUAS with respect to
\(\varepsilon=k/\omega\) and gives
\eqref{eq:main-SPUAS-bound} with \(x_a\) replaced by
\(x_{a,\kappa}\).
\end{proof}

%\section*{References}
\bibliographystyle{IEEEtran}
\bibliography{IEEEabrv,references}

\begin{thebibliography}{10}
\providecommand{\url}[1]{#1}
\csname url@samestyle\endcsname
\providecommand{\newblock}{\relax}
\providecommand{\bibinfo}[2]{#2}
\providecommand{\BIBentrySTDinterwordspacing}{\spaceskip=0pt\relax}
\providecommand{\BIBentryALTinterwordstretchfactor}{4}
\providecommand{\BIBentryALTinterwordspacing}{\spaceskip=\fontdimen2\font plus
\BIBentryALTinterwordstretchfactor\fontdimen3\font minus
  \fontdimen4\font\relax}
\providecommand{\BIBforeignlanguage}[2]{{%
\expandafter\ifx\csname l@#1\endcsname\relax
\typeout{** WARNING: IEEEtran.bst: No hyphenation pattern has been}%
\typeout{** loaded for the language `#1'. Using the pattern for}%
\typeout{** the default language instead.}%
\else
\language=\csname l@#1\endcsname
\fi
#2}}
\providecommand{\BIBdecl}{\relax}
\BIBdecl

\bibitem{krstic2000stability}
M.~Krstic and H.-H. Wang, ``Stability of extremum seeking feedback for general
  nonlinear dynamic systems,'' \emph{Automatica}, vol.~36, no.~4, pp. 595--601,
  2000.

\bibitem{ariyur2003real}
K.~B. Ariyur and M.~Krstic, \emph{Real-time optimization by extremum-seeking
  control}.\hskip 1em plus 0.5em minus 0.4em\relax John Wiley \& Sons, 2003.

\bibitem{tan2006non}
Y.~Tan, D.~Ne{\v{s}}i{\'c}, and I.~Mareels, ``On non-local stability properties
  of extremum seeking control,'' \emph{Automatica}, vol.~42, no.~6, pp.
  889--903, 2006.

\bibitem{tan2009global}
Y.~Tan, D.~Ne{\v{s}}i{\'c}, I.~M. Mareels, and A.~Astolfi, ``On global extremum
  seeking in the presence of local extrema,'' \emph{Automatica}, vol.~45,
  no.~1, pp. 245--251, 2009.

\bibitem{ghaffari2012multivariable}
A.~Ghaffari, M.~Krsti{\'c}, and D.~Ne{\v{s}}i{\'c}, ``Multivariable
  newton-based extremum seeking,'' \emph{Automatica}, vol.~48, no.~8, pp.
  1759--1767, 2012.

\bibitem{liu2015stochastic}
S.-J. Liu and M.~Krstic, ``Stochastic averaging in discrete time and its
  applications to extremum seeking,'' \emph{IEEE Transactions on Automatic
  control}, vol.~61, no.~1, pp. 90--102, 2015.

\bibitem{williams2026local}
A.~Williams, M.~Krstic, and A.~Scheinker, ``Local practically safe extremum
  seeking with assignable rate of attractivity to the safe set,''
  \emph{Automatica}, vol. 183, p. 112611, 2026.

\bibitem{williams2024semiglobal}
------, ``Semiglobal safety-filtered extremum seeking with unknown cbfs,''
  \emph{IEEE Transactions on Automatic Control}, vol.~70, no.~3, pp.
  1698--1713, 2024.

\bibitem{williams2026generalized}
A.~Williams, J.~Cort{\'e}s, and A.~Scheinker, ``Generalized multi-constraint
  extremum seeking,'' in \emph{2026 American Control Conference (ACC)}.\hskip
  1em plus 0.5em minus 0.4em\relax IEEE, 2026, pp. 1342--1349.

\bibitem{teel2001solving}
A.~R. Teel and D.~Popovic, ``Solving smooth and nonsmooth multivariable
  extremum seeking problems by the methods of nonlinear programming,'' in
  \emph{Proceedings of the 2001 American Control Conference.(Cat. No.
  01CH37148)}, vol.~3.\hskip 1em plus 0.5em minus 0.4em\relax IEEE, 2001, pp.
  2394--2399.

\bibitem{poveda2017framework}
J.~I. Poveda and A.~R. Teel, ``A framework for a class of hybrid extremum
  seeking controllers with dynamic inclusions,'' \emph{Automatica}, vol.~76,
  pp. 113--126, 2017.

\bibitem{poveda2021robust}
J.~I. Poveda and N.~Li, ``Robust hybrid zero-order optimization algorithms with
  acceleration via averaging in time,'' \emph{Automatica}, vol. 123, p. 109361,
  2021.

\bibitem{grushkovskaya2017extremum}
V.~Grushkovskaya, H.-B. D{\"u}rr, C.~Ebenbauer, and A.~Zuyev, ``Extremum
  seeking for time-varying functions using lie bracket approximations,''
  \emph{IFAC-PapersOnLine}, vol.~50, no.~1, pp. 5522--5528, 2017.

\bibitem{mimmo2024extremum}
N.~Mimmo, G.~Carnevale, A.~Testa, and G.~Notarstefano, ``Extremum seeking
  tracking for derivative-free distributed optimization,'' \emph{IEEE
  Transactions on Control of Network Systems}, vol.~12, no.~1, pp. 584--595,
  2024.

\bibitem{tsubakino2023extremum}
D.~Tsubakino, T.~R. Oliveira, and M.~Krstic, ``Extremum seeking for distributed
  delays,'' \emph{Automatica}, vol. 153, p. 111044, 2023.

\bibitem{poveda2021nonsmooth}
J.~I. Poveda and M.~Krsti{\'c}, ``Nonsmooth extremum seeking control with
  user-prescribed fixed-time convergence,'' \emph{IEEE Transactions on
  Automatic Control}, vol.~66, no.~12, pp. 6156--6163, 2021.

\bibitem{yilmaz2024prescribed}
C.~T. Yilmaz and M.~Krstic, ``Prescribed-time extremum seeking for delays and
  pdes using chirpy probing,'' \emph{IEEE Transactions on Automatic Control},
  vol.~69, no.~11, pp. 7710--7725, 2024.

\bibitem{grushkovskaya2024step}
V.~Grushkovskaya and C.~Ebenbauer, ``Step-size rules for lie bracket-based
  extremum seeking with asymptotic convergence guarantees,'' \emph{IEEE Control
  Systems Letters}, vol.~8, pp. 1967--1972, 2024.

\bibitem{mimmo2024uniform}
N.~Mimmo, L.~Marconi, and G.~Notarstefano, ``Uniform nonconvex optimization via
  extremum seeking,'' \emph{IEEE Transactions on Automatic Control}, vol.~69,
  no.~12, pp. 8263--8276, 2024.

\bibitem{pokhrel2026higher}
S.~Pokhrel and S.~A. Eisa, ``Higher-order lie bracket approximation and
  averaging of control-affine systems with application to extremum seeking,''
  \emph{Automatica}, vol. 188, p. 112950, 2026.

\bibitem{frihauf2012nash}
P.~Frihauf, M.~Krsti{\'c}, and T.~Ba{\c{s}}ar, ``Nash equilibrium seeking in
  noncooperative games,'' \emph{IEEE Transactions on Automatic Control},
  vol.~57, no.~5, pp. 1192--1207, 2012.

\bibitem{stankovic2012distributed}
M.~S. Stankovi{\'c}, K.~H. Johansson, and D.~M. Stipanovi{\'c}, ``Distributed
  seeking of nash equilibria with applications to mobile sensor networks,''
  \emph{IEEE Transactions on Automatic Control}, vol.~57, no.~4, pp. 904--919,
  2012.

\bibitem{ratto2026nested}
B.~Ratto, A.~Williams, M.~Krsti{\'c}, T.~Ba{\c{s}}ar, and A.~Scheinker,
  ``Nested extremum seeking converges to {Stackelberg} equilibrium,'' in
  \emph{Proc. 65th IEEE Conf. Decision Control (CDC)}, Honolulu, HI, USA, 2026,
  to appear; also available as arXiv:2603.24756.

\bibitem{scheinker2024100}
A.~Scheinker, ``100 years of extremum seeking: A survey,'' \emph{Automatica},
  vol. 161, p. 111481, 2024.

\bibitem{durr2013lie}
H.-B. D{\"u}rr, M.~S. Stankovi{\'c}, C.~Ebenbauer, and K.~H. Johansson, ``Lie
  bracket approximation of extremum seeking systems,'' \emph{Automatica},
  vol.~49, no.~6, pp. 1538--1552, 2013.

\bibitem{scheinker2014non}
A.~Scheinker and M.~Krsti{\'c}, ``Non-c 2 lie bracket averaging for nonsmooth
  extremum seekers,'' \emph{Journal of Dynamic Systems, Measurement, and
  Control}, vol. 136, no.~1, p. 011010, 2014.

\bibitem{suttner2023nonsmooth}
R.~Suttner, ``Nonsmooth optimization by lie bracket approximations into random
  directions,'' \emph{Systems \& Control Letters}, vol. 174, p. 105481, 2023.

\bibitem{suttner2024overcoming}
R.~Suttner and M.~Krsti{\'c}, ``Overcoming local extrema in torque-actuated
  source seeking using the divergence theorem and delay,'' \emph{Automatica},
  vol. 167, p. 111799, 2024.

\bibitem{suttner2026non}
R.~Suttner, C.~Ebenbauer, and S.~Dashkovskiy, ``Non-local extremum seeking
  based on the divergence theorem,'' \emph{arXiv preprint arXiv:2603.01200},
  2026.

\bibitem{abdelfattah2026nonsmooth}
H.~Abdelfattah, S.~A. Eisa, and P.~Stechlinski, ``Nonsmooth high-order
  averaging theory with application to extremum seeking optimization and
  control,'' \emph{arXiv preprint arXiv:2606.00969}, 2026.

\bibitem{scheinker2013extremum}
A.~Scheinker, M.~Bland, M.~Krsti{\'c}, and J.~Audia, ``Extremum seeking-based
  optimization of high voltage converter modulator rise-time,'' \emph{IEEE
  Transactions on Control Systems Technology}, vol.~22, no.~1, pp. 34--43,
  2013.

\bibitem{liu2020primer}
S.~Liu, P.-Y. Chen, B.~Kailkhura, G.~Zhang, A.~O. Hero~III, and P.~K. Varshney,
  ``A primer on zeroth-order optimization in signal processing and machine
  learning: Principals, recent advances, and applications,'' \emph{IEEE Signal
  Processing Magazine}, vol.~37, no.~5, pp. 43--54, 2020.

\bibitem{flaxman2004online}
A.~D. Flaxman, A.~T. Kalai, and H.~B. McMahan, ``Online convex optimization in
  the bandit setting: gradient descent without a gradient,'' \emph{arXiv
  preprint cs/0408007}, 2004.

\bibitem{folland1999real}
G.~B. Folland, \emph{Real analysis: modern techniques and their
  applications}.\hskip 1em plus 0.5em minus 0.4em\relax John Wiley \& Sons,
  1999.

\bibitem{evans1991measure}
L.~C. Evans and R.~F. Gariepy, \emph{Measure theory and fine properties of
  functions}.\hskip 1em plus 0.5em minus 0.4em\relax CRC press, 1991, vol.~5.

\bibitem{sanders2007averaging}
J.~A. Sanders, F.~Verhulst, and J.~Murdock, \emph{Averaging methods in
  nonlinear dynamical systems}.\hskip 1em plus 0.5em minus 0.4em\relax
  Springer, 2007, vol.~59.

\bibitem{bailleul2022explicit}
A.~Bailleul, ``Explicit kronecker--weyl theorems and applications to prime
  number races,'' \emph{Research in Number Theory}, vol.~8, no.~3, p.~43, 2022.

\bibitem{moreau2000practical}
L.~Moreau and D.~Aeyels, ``Practical stability and stabilization,'' \emph{IEEE
  Transactions on Automatic Control}, vol.~45, no.~8, pp. 1554--1558, 2000.

\bibitem{nocedal2006numerical}
J.~Nocedal and S.~J. Wright, \emph{Numerical optimization}.\hskip 1em plus
  0.5em minus 0.4em\relax Springer, 2006.

\bibitem{boyd2004convex}
S.~Boyd and L.~Vandenberghe, \emph{Convex Optimization}.\hskip 1em plus 0.5em
  minus 0.4em\relax Cambridge, UK: Cambridge University Press, 2004.

\bibitem{rudin1976principles}
W.~Rudin, \emph{Principles of Mathematical Analysis}, 3rd~ed.\hskip 1em plus
  0.5em minus 0.4em\relax New York: McGraw-Hill, 1976.

\end{thebibliography}

\end{document}